\documentclass[12pt]{elsarticle}
  \usepackage{graphicx}
  \usepackage{latexsym,amsmath,amssymb,epsfig,amsthm}
\usepackage[enableskew,vcentermath]{youngtab}

  \def\CC{{\mathbb C}}

  \def\NN{{\mathbb N}}

  \def\ZZ{{\mathbb Z}}

  \def\aA{{\mathcal A}}
  \def\AB{{\mathcal A}^B}
  \def\wAB{\widehat{\mathcal A}}
  \def\bB{{\mathcal B}}
  \def\cC{{\mathcal C}}
  \def\dD{{\mathcal D}}
  \def\eE{{\mathcal E}}
  \def\iI{{\mathcal I}}
  \def\kK{{\mathcal K}}
  \def\lL{{\mathcal L}}
  \def\mM{{\mathcal M}}
  \def\rR{{\mathcal R}}
  \def\veps{{\varepsilon}}
  \def\teps{\tilde{\veps}}

  \def\barr{{\rm bar}}
  
  \def\Comp{{\rm Comp}}
  \def\cDes{{\rm sDes}}
  \def\Des{{\rm Des}}
  
  \def\wDes{{\rm wDes}}
  \def\hA{\widehat{A}}
  \def\ch{{\rm ch}}
  
  \def\finv{{\rm finv}}
  \def\fmaj{{\rm fmaj}}
  
  \def\QSym{{\rm QSym}}
  
  \def\inv{{\rm inv}}
  \def\maj{{\rm maj}}
  \def\Neg{{\rm Neg}}
  
  \def\SYT{{\rm SYT}}
  \def\wei{{\rm wt}}
  
  \def\sm{\smallsetminus}

  \theoremstyle{plain}
    \newtheorem{theorem}{Theorem}[section]
    \newtheorem{proposition}[theorem]{Proposition}
    \newtheorem{lemma}[theorem]{Lemma}
    \newtheorem{corollary}[theorem]{Corollary}
    
    \theoremstyle{definition}
    \newtheorem{definition}[theorem]{Definition}
    \newtheorem{example}[theorem]{Example}

    \newtheorem{remark}[theorem]{Remark}
    \newtheorem{problem}[theorem]{Problem}

  \numberwithin{equation}{section}

\begin{document}

\title{Character formulas and descents for the hyperoctahedral group}

  \author[BI]{Ron~M.~Adin}
   \ead{radin@math.biu.ac.il}

  \author[UA]{Christos~A.~Athanasiadis}
  \ead{caath@math.uoa.gr}

  \author[DC]{Sergi~Elizalde\corref{cor1}}
  \ead{sergi.elizalde@dartmouth.edu}

  \author[BI]{Yuval~Roichman}
  \ead{yuvalr@math.biu.ac.il}

 \address[BI]{Department of Mathematics, Bar-Ilan University, Ramat-Gan 52900, Israel}

\address[UA]{Department of Mathematics, National and Kapodistrian University of Athens, Panepistimioupolis, Athens 15784, Greece}
  
\address[DC]{Department of Mathematics, Dartmouth College, Hanover, NH 03755, USA}

\cortext[cor1]{Corresponding author}

\date{submitted: August 16, 2016}

\begin{abstract}
A general setting to study a certain type of formulas, expressing
characters of the symmetric group $\mathfrak{S}_n$ explicitly in
terms of descent sets of combinatorial objects, has been developed
by two of the authors. This theory is further investigated in this
paper and extended to the hyperoctahedral group $B_n$. Key
ingredients are a new formula for the irreducible characters of
$B_n$, the signed quasisymmetric functions introduced by Poirier,
and a new family of matrices of Walsh--Hadamard type. Applications
include formulas for natural $B_n$-actions on coinvariant and
exterior algebras and on the top homology of a certain poset in
terms of the combinatorics of various classes of signed
permutations, as well as a $B_n$-analogue of an equidistribution
theorem of D\'esarm\'enien and Wachs.
\end{abstract}

\maketitle


\section{Introduction}
\label{sec:into}

One of the main goals of combinatorial representation theory, as
described in the survey article~\cite{BR99}, is to provide
explicit formulas which express the values of characters of
interesting representations as weighted enumerations of nice 
combinatorial objects. Perhaps the best known example of such a
formula is the Murnaghan--Nakayama rule~\cite[page~117]{Mac95}
\cite[Section~7.17]{StaEC2} for the irreducible characters of the
symmetric group $\mathfrak{S}_n$.

Several such formulas of a more specific type, expressing
characters of the symmetric groups and their Iwahori--Hecke
algebras in terms of the distribution of the descent set over
classes of permutations, or other combinatorial objects, have been
discovered in the past two decades. The prototypical example is
Roichman's rule~\cite{Roi97} for the irreducible characters of
$\mathfrak{S}_n$ (and the corresponding Hecke algebra), where the
enumerated objects are either Knuth classes of permutations, or
standard Young tableaux. Other notable examples include the
character of the Gelfand model (i.e., the  multiplicity-free sum
of all irreducible characters)~\cite{APR08}, characters of
homogenous components of the coinvariant algebra~\cite{APR01}, Lie
characters~\cite{GR93}, characters of Specht modules of zigzag
shapes~\cite{Ge84}, characters induced from a lower ranked
exterior algebra~\cite{ER14} and $k$-root
enumerators~\cite{Roi14}, determined by the distribution of the
descent set over involutions, elements of fixed Coxeter length,
conjugacy classes, inverse descent classes, arc permutations and
$k$-roots of the identity permutation, respectively. The formulas
in question evaluate these characters by $\{-1, 0, 1\}$-weighted
enumerations of the corresponding classes of permutations, where
exactly the same weight function appears in all summations. Some
new examples are given in Sections~\ref{sec:conj}--\ref{sec:arc}
of this paper.

An abstract framework for this phenomenon, which captures all
aforementioned examples, was proposed in~\cite{AR15}. Characters
which are expressed by such formulas (see
Definition~\ref{def:fineA}) are called \emph{fine characters}, and
classes which carry them are called \emph{fine sets}. It was shown
in~\cite{AR15} that the equality of two fine characters is
equivalent to the equidistribution of the descent set over the
corresponding fine sets. This implies, in particular, the
equivalence of classical theorems of Lusztig--Stanley in invariant
theory~\cite{St79} and Foata--Sch\"utzenberger in permutation
statistics~\cite{FS78}. Furthermore, it was shown in~\cite{AR15}
that fine sets can be characterized by the symmetry and
Schur-positivity of the associated quasisymmetric functions. For
the latter, the reader is referred to Gessel--Reutenauer's seminal
paper~\cite{GR93}.

This paper investigates this setting further and provides a
nontrivial extension to the hyperoctahedral group $B_n$.
Section~\ref{subsec:main_A} gives a more explicit version of the
main result of~\cite{AR15}. This version (Theorem~\ref{thm:mainA})
states that a given $\mathfrak{S}_n$-character is carried by a
fine set $\bB$ if and only if its Frobenius characteristic is
equal to the quasisymmetric generating function of the descent set
over $\bB$. To extend this result to the group $B_n$, suitable
signed analogues of the concepts of fine characters and fine sets
have to be introduced (see Definition~\ref{def:fineB}) and
suitable signed analogues of the fundamental quasisymmetric
functions, namely those defined and studied by
Poirier~\cite{Poi98}, are employed. For the former task, a signed
analogue of the concept of descent set is used (see
Section~\ref{subsec:perm}) and an analogue of Roichman's rule for
the irreducible characters of $B_n$ (Theorem~\ref{thm:roiB}) is
proven. The weights involved in this rule are used to define type
$B$ fine sets as those sets whose weighted enumeration determines
a non-virtual $B_n$-character. The main result
(Theorem~\ref{thm:main}) of this paper states that a given
$B_n$-character is carried by a fine set $\bB$ if and only if its
Frobenius characteristic is equal to the quasisymmetric generating
function of the signed descent set over $\bB$.

The proof of one direction of Theorem~\ref{thm:main} amounts to
the invertibility of a new family of matrices of Walsh--Hadamard
type, introduced in Section~\ref{subsec:weight_matrix}. This is
shown by a tricky computation of their determinants.
Equidistribution and Schur-positivity phenomena, similar to those
in~\cite{AR15}, are easily derived from Theorem~\ref{thm:main}
(see Corollary~\ref{cor:main}).

Sections~\ref{sec:Knuth}--\ref{sec:arc} show that nearly all
examples of fine $\mathfrak{S}_n$-characters mentioned earlier
have natural $B_n$-analogues. These include the irreducible
characters of $B_n$, the Gelfand model, the characters of the
natural $B_n$-action on the homogeneous components of the
coinvariant algebra of type $B$, signed analogues of the Lie
characters, characters induced from exterior algebras and $k$-root
enumerators, with corresponding fine sets consisting of elements
of Knuth classes of type $B_n$, involutions, signed permutations
of fixed flag-inversion number or flag-major index, conjugacy
classes in $B_n$, signed analogues of arc permutations and
$k$-roots of the identity signed permutation, respectively.
Theorem~\ref{thm:main} implies explicit formulas for these
characters in terms of the distribution of the signed descent set
over the corresponding fine sets. Section~\ref{subsec:derange} is
concerned with the character of the $\mathfrak{S}_n$-action on the
top homology of the poset of injective words, studied by Reiner
and Webb~\cite{ReiW04}, and with its natural $B_n$-analogue. This
allows us to interpret a theorem of D\'esarm\'enien and
Wachs~\cite{DeWa88, DeWa93} on the equidistribution of the descent
set over derangements and desarrangements in $\mathfrak{S}_n$ in
the language of $\mathfrak{S}_n$-fine sets, and to derive a
$B_n$-analogue. The latter task provided much of the motivation
behind this paper.

The structure of this paper is as follows. Section~\ref{sec:back}
reviews background material from combinatorial representation
theory and sets up the notation on permutations, partitions, Young
tableaux, compositions, descent sets and their signed analogues
which is necessary to define fine sets. The main definitions and
results are stated in Section~\ref{sec:main}, which also derives
the aforementioned result on $\mathfrak{S}_n$-fine sets
(Theorem~\ref{thm:mainA}) from the main result of~\cite{AR15}.
Section~\ref{sec:proof} states and proves the $B_n$-analogue of
Roichman's rule (Theorem~\ref{thm:roiB}) and uses this result,
together with methods from linear algebra, to prove the main
result (Theorem~\ref{thm:main}) on $B_n$-fine sets. As already
mentioned, Sections~\ref{sec:Knuth}--\ref{sec:arc} discuss
applications and examples. Section~\ref{sec:rem} concludes with
remarks and open problems.

\section{Background and notation}
\label{sec:back}

This section fixes notation and briefly reviews background
material regarding the combinatorics of (signed) permutations and
Young (bi)tableaux, the representation theory of the symmetric and
hyperoctahedral groups and the theory of symmetric and
quasisymmetric functions which will be needed in the sequel. More
information on these topics and any undefined terminology can be
found in~\cite[Chapter~I]{Mac95} \cite[Chapter~7]{StaEC2}
\cite{SW85} \cite{Ste89}.

Throughout this paper, $x = (x_1, x_2, \ldots)$ and $y = (y_1,
y_2,\dots)$ will be sequences of pairwise commuting
indeterminates. For positive integers $n$ we set $[n] := \{1,
2,\dots,n\}$ and $\Omega_n := \{1, 2,\dots,n\} \cup \{\bar{1},
\bar{2},\dots,\bar{n}\}$. The absolute value $|a|$ of $a \in
\Omega_n$ is the element of $[n]$ obtained from $a$ by simply
forgetting the bar, if present. We will consider the set $[n]$
totally ordered by its natural order $<$ and $\Omega_n$ totally
ordered by the order
  \begin{equation} \label{eq:total}
    \bar{1} <_r \bar{2} <_r \cdots <_r \bar{n} <_r 1 <_r 2 <_r \cdots
    <_r n,
  \end{equation}
corresponding to the right lexicographic order on the set $[n]
\times \{-, +\}$, and endowed with the standard involution $\,
\bar{}: \Omega_n \to \Omega_n$, mapping $a$ to $\bar{a}$ and vice
versa, for every $a \in [n]$.

\subsection{Compositions and partitions}
\label{subsec:comp}

A \emph{composition} of a positive integer $n$, written as $\alpha
\vDash n$, is a sequence $\alpha = (\alpha_1,
\alpha_2,\dots,\alpha_k)$ of positive integers (parts) summing to
$n$. We will write $S(\alpha) = \{r_1, r_2,\dots,r_k\}$ for the
set of partial sums $r_i = \alpha_1 + \alpha_2 + \cdots +
\alpha_i$ $(1 \le i \le k)$ of such $\alpha$, and will denote by
$\Comp(n)$ the set of all compositions of $n$. Clearly, the map
which assigns $S(\alpha)$ to $\alpha$ is a bijection from
$\Comp(n)$ to the set of all subsets of $[n]$ which contain $n$. A
\emph{partition} of $n$, written as $\lambda \vdash n$, is a
composition $\lambda$ of $n$ whose parts appear in a weakly
decreasing order. We shall also consider the empty composition
(and partition) of $n = 0$.

A \emph{bipartition} of $n$, written as $(\lambda, \mu) \vdash n$,
is a pair $(\lambda, \mu)$ of (possibly empty) partitions such
that $\lambda \vdash k$ and $\mu \vdash n-k$ for some $0 \le k \le
n$.

A \emph{signed composition} of $n$ is a composition $\gamma$ of
$n$ some of the parts of which may be barred. Such a $\gamma$ can
be encoded by its set of partial sums $S(\gamma) = \{s_1,
s_2,\dots,s_k\}$, where $s_1 < s_2 < \cdots < s_k = n$, together
with the sign map $\veps: S(\gamma) \to \{-, +\}$ defined by
setting $\veps(s_i) = +$ if the $i$th part of $\gamma$ is unbarred
and $\veps(s_i) = -$ otherwise. Clearly, the map which assigns the
pair $(S(\gamma), \veps)$ to $\gamma$ is a bijection from the set
$\Comp^B(n)$ of all signed compositions of $n$ to the set
$\Sigma^B(n)$ consisting of all pairs $(S, \veps)$, where $S$ is a
subset of $[n]$ which contains $n$ and $\veps: S \to \{-, +\}$ is
a map. Thus, as observed for instance in \cite[Section~3.2]{HP10},
the total number of signed compositions of $n$ is equal to $2
\cdot 3^{n-1}$. We will consider the elements of $\Sigma^B(n)$,
which we think of as ``signed sets", instead of the corresponding
signed compositions, whenever this is notationally convenient.

Given $\sigma = (S, \veps) \in \Sigma^B(n)$, with $S = \{s_1,
s_2,\dots,s_k\}$ and $s_1 < s_2 < \cdots < s_k = n$, we may extend
$\veps$ to a map $\teps: [n] \to \{-, +\}$ by setting $\teps(j) :=
\veps(s_i)$ for all $s_{i-1} < j \le s_i$ $(1 \le i \le k)$, where
$s_0 := 0$. We will refer to $\teps$ as the \emph{sign vector} of
$\sigma$ (and of the corresponding signed composition). The
following notation will be useful in Section~\ref{subsec:quasi}.
\begin{definition} \label{def:wDes}
For $\sigma = (S, \veps) \in \Sigma^B(n)$ we denote by
$\wDes(\sigma)$ the set of elements $s_i \in S$, other than $s_k =
n$, for which either $\veps(s_i) = \veps(s_{i+1})$, or else
$\veps(s_i) = +$ and $\veps(s_{i+1}) = -$.
\end{definition}
For example, if $n=9$ and $\gamma = (2, \bar{1}, \bar{2}, 1, 3)$,
then for the signed set $\sigma = (S, \veps)$ corresponding to
$\gamma$ we have $S = \{2, 3, 5, 6, 9\}$, $\veps(2) = \veps(6) =
\veps(9) = +$, $\veps(3) = \veps(5) = -$, $\wDes(\sigma) = \{2, 3,
6\}$ and sign vector $\teps = (+, +, -, -, -, +, +, +, +)$.

\subsection{Permutations and tableaux}
\label{subsec:perm}

We will denote by $\mathfrak{S}_n$ the symmetric group of all
permutations of the set $[n]$, namely bijective maps $w: [n] \to
[n]$, and by $\SYT(\lambda)$ the set of all standard Young
tableaux of shape $\lambda \vdash n$. We recall that the
Robinson--Schensted correspondence is a bijection of fundamental
importance from $\mathfrak{S}_n$ to the set of pairs $(P, Q)$ of
standard Young tableaux of the same shape and size $n$. The
\emph{descent set} of a permutation $w \in \mathfrak{S}_n$ is
defined as $\Des(w) := \{i \in [n-1]: w(i) > w(i+1)\}$. The
\emph{descent set} $\Des(Q)$ of a standard Young tableau $Q \in
\SYT(\lambda)$, where $\lambda \vdash n$, is the set of all $i \in
[n-1]$ for which $i+1$ appears in a lower row in $Q$ than $i$
does. A basic property of the Robinson--Schensted correspondence
asserts that $\Des(w) = \Des(Q(w))$, where $(P(w), Q(w))$ is the
pair of tableaux associated to $w \in \mathfrak{S}_n$.

The hyperoctahedral group $B_n$ consists of all signed
permutations of length $n$, meaning bijective maps $w: \Omega_n
\to \Omega_n$ such that $w(a) = b \Rightarrow w(\bar{a}) =
\bar{b}$ for every $a \in \Omega_n$. Unless indicated otherwise,
we will think of a signed permutation $w \in B_n$ as the sequence of values $(w(1), w(2),\dots,w(n))$, 
namely as a permutation in $\mathfrak{S}_n$, written in one line notation, with some of its entries barred. The (unsigned) descent set of
$w\in B_n$ is defined as the descent set of the sequence
$(w(1),w(2), \ldots, w(n))$ with respect to the total order
\eqref{eq:total}, namely $\Des(w) := \{i \in [n-1]: w(i) >_r
w(i+1)\}$.
\begin{definition}\label{cDes}
The \emph{signed} (or \emph{colored}) \emph{descent set} of $w \in B_n$, denoted $\cDes(w)$, is the signed set $(S, \varepsilon) \in \Sigma^B (n)$ defined as follows:
  \begin{itemize}
    \item[$\bullet$]
    The set $S$ consists of $n$ along with all $s \in [n-1]$ for which
    either $w(s) >_r w(s+1)$ or $w(s)$ is barred and $w(s+1)$ is unbarred.
    \item[$\bullet$]
    For every $s \in S$, $\veps(s) = -$ if $w(s)$ is barred and $\veps(s) = +$ otherwise.
  \end{itemize}
\end{definition}
In other words, denoting $|w| = (|w(1)|, \ldots, |w(n)|) \in \mathfrak{S}_n$ and defining $\veps_w : [n] \to \{-,+\}$ by $\veps_w(i) = -$ if $w(i)$ is barred and $\veps_w(i) = +$ otherwise,
the set $S$ is the union of $\Des(|w|)$, $\{n\}$ and 
$\{i \in [n-1] \,:\, \veps_w(i) \ne \veps_w(i+1)\}$,
while $\veps$ is the restriction of $\veps_w$ to $S$.
For example, if $w = (\bar{5}, \bar{2}, \bar{8}, 1, 3, 9, 4, \bar{6}, \bar{7}) \in B_9$ then
$\cDes(w) = (S, \veps)$ with $S = \{1, 3, 6, 7, 9\}$ and $\veps(1) = \veps(3) = \veps(9) = -$, $\veps(6) = \veps(7) = +$.
We note that if $\sigma = \cDes(w)$ for
some $w \in B_n$ then $\wDes(\sigma) = \Des(w) := \{i \in [n-1]:
w(i) >_r  w(i+1)\}$, where $\wDes(\sigma)$ is as in
Definition~\ref{def:wDes}.

Given a bipartition $(\lambda, \mu)$ of $n$, a \emph{standard
Young bitableau} of shape $(\lambda, \mu)$ and size $n$ is a pair $Q = (Q^+,
Q^-)$ of tableaux which are (strictly) increasing along rows and
columns and have the following properties: (a) $Q^+$ has shape
$\lambda$; (b) $Q^-$ has shape $\mu$; and (c) every element of
$[n]$ appears (exactly once) in either $Q^+$ or $Q^-$. We will
denote by $\SYT(\lambda, \mu)$ the set of all standard Young
bitableaux of shape $(\lambda, \mu)$.

The \emph{descent set}, denoted
by $\Des(Q)$, of a bitableau $Q = (Q^+, Q^-) \in \SYT(\lambda,
\mu)$ consists of all 
$i\in [n-1]$, such that $i$ and $i+1$ appear in the same tableau ($Q^+$ or $Q^-$) and
    $i+1$ appears in a lower row than $i$, or $i$ appears in $Q^+$ and $i+1$ appears in 
$Q^-$.

\begin{definition}
The \emph{signed} (or \emph{colored}) \emph{descent set}, denoted
by $\cDes(Q)$, of a bitableau $Q = (Q^+, Q^-) \in \SYT(\lambda,
\mu)$ is the element $(S, \veps) \in \Sigma^B (n)$ defined as
follows:
  \begin{itemize}
    \item[$\bullet$]
    The set $S$ consists of $n$ along with all $s \in [n-1]$ for which
    either $s$ and $s+1$ appear in the same tableau ($Q^+$ or $Q^-$) and
    $s+1$ appears in a lower row than $s$,
    or $s$ and $s+1$ appear in different tableaux.
    \item[$\bullet$]
    For every $s \in S$, $\veps(s) = +$ if $s$ appears in $Q^+$ and $\veps(s) = -$ otherwise.
  \end{itemize}
\end{definition}
Clearly, for the corresponding sign vector $\teps: [n] \to \{-,
+\}$ we have, for every $i \in [n]$, $\teps(i) = +$ if $i$ appears
in $Q^+$ and $\teps(i) = -$ if $i$ appears in $Q^-$.

The Robinson--Schensted correspondence has a natural
$B_n$-analogue; see, for instance, \cite[Section~6]{Sta82}
\cite{SW85} and Section~\ref{sec:Knuth}. This analogue is a
bijection from $B_n$ to the set of pairs $(P, Q)$ of standard
Young bitableaux of the same shape and size $n$. It has the
property that $\cDes(w) = \cDes(Q^B(w))$, where $(P^B(w), Q^B(w))$
is the pair associated to $w \in B_n$.

\subsection{Characters and symmetric functions}
\label{subsec:char}

Conjugacy classes in $\mathfrak{S}_n$ consist of permutations of
given cycle type $\lambda \vdash n$. As a result, the (complex)
irreducible characters of $\mathfrak{S}_n$ are indexed by
partitions of $n$. We will denote by $\chi^\lambda$ the
irreducible character corresponding to $\lambda \vdash n$. For a
class function $\chi: \mathfrak{S}_n \to \CC$ and composition
$\alpha \models n$, we will write $\chi(\alpha)$ for the value of
$\chi$ on any element of $\mathfrak{S}_n$ whose cycle type is (the
decreasing rearrangement of) $\alpha$.

We will denote by $\Lambda(x)$ the $\CC$-algebra of symmetric
functions in $x$. The algebra $\Lambda(x)$ encodes the
combinatorics of characters of the symmetric groups via the
Frobenius characteristic map, defined by
  \begin{equation} \label{eq:chA}
    \ch(\chi) \ = \ \frac{1}{n!} \, \sum_{w \in \mathfrak{S}_n} \chi(w) p_w(x)
  \end{equation}
where $\chi: \mathfrak{S}_n \to \CC$ is a class function, $p_w(x)
= p_\alpha (x)$ for every permutation $w \in \mathfrak{S}_n$ of
cycle type $\alpha \vdash n$, and $p_\alpha(x)$ is a power sum
symmetric function. This map is a $\CC$-linear isomorphism from
the space of such class functions on $\mathfrak{S}_n$ to the
degree $n$ homogeneous part $\Lambda^n (x)$ of $\Lambda(x)$, and
satisfies $\ch(\chi^\lambda) =  s_\lambda(x)$ for every $\lambda
\vdash n$; see \cite[Section~7.18]{StaEC2} for a detailed
discussion and further information.

The values $\chi^\lambda(\alpha)$ of the irreducible characters of
$\mathfrak{S}_n$ appear in the Frobenius formula
  \begin{equation} \label{eq:FrobA}
    p_\alpha (x) \ = \ \sum_{\lambda \vdash n} \chi^\lambda(\alpha) s_\lambda(x)
  \end{equation}
for $\mathfrak{S}_n$, expressing $p_\alpha(x)$ as a linear
combination of Schur functions.

We now recall an explicit formula for $\chi^\lambda(\alpha)$. Let
$\alpha = (\alpha_1, \alpha_2, \ldots, \alpha_k) \in \Comp(n)$ be
a composition of $n$ with $S(\alpha) = \{r_1, r_2,\dots,r_k\}$, as
in Section~\ref{subsec:comp} (so that $r_k = n$), and set $r_0 =
0$. A subset $S$ of $[n-1]$ is said to be \emph{$\alpha$-unimodal}
(or \emph{unimodal with respect to $\alpha$})
\cite[Definition~3.1]{AR15} if the intersection of $S$ with each
of the sets $\{r_{i-1} + 1,\dots,r_i - 1\}$ $(1 \le i \le k)$ is a
prefix (possibly empty) of the latter. For instance, if $\alpha =
(n)$, then the $\alpha$-unimodal subsets of $[n-1]$ are those
equal to $[p]$ for some $p \in \{0, 1,\dots,n-1\}$; and if $\alpha
= (1, 1,\dots,1)$, then every subset of $[n-1]$ is
$\alpha$-unimodal. As another example, if $\alpha = (3, 1, 4, 2)$,
then $\{1, 3, 5, 6\}$ is $\alpha$-unimodal but $\{1, 3, 5, 7\}$ is
not. We will denote by $U_\alpha$ the set of $\alpha$-unimodal
subsets of $[n-1]$.

The following theorem is a special case of
\cite[Theorem~4]{Roi97}; see also \cite[Section~I.2]{BR99} and
references therein. A direct combinatorial proof appeared
in~\cite{Ram98}.

\begin{theorem} {\rm (\cite{Roi97})} \label{thm:roi}
For all partitions $\lambda \vdash n$ and compositions $\alpha
\models n$,
  \begin{equation} \label{eq:chiroi}
    \chi^\lambda (\alpha) \ =
              \sum_{\substack{Q \in \SYT(\lambda) \\ \Des(Q) \in U_\alpha}}
              (-1)^{|\Des(Q) \sm S(\alpha)|}.
  \end{equation}
\end{theorem}

The concept of $\alpha$-unimodality, and Theorem~\ref{thm:roi} in
particular, were recently applied to prove conjectures of Regev
concerning induced characters~\cite[Section~9]{ER14} and of
Shareshian and Wachs concerning chromatic quasisymmetric
functions~\cite{Ath15}.

\medskip
We now briefly describe the analogue of this theory for the group
$B_n$. Conjugacy classes in $B_n$, and hence (complex) irreducible
$B_n$-characters, are in one-to-one correspondence with
bipartitions of $n$. More precisely, each element $w \in B_n$,
viewed as a permutation of the set $\Omega_n$, can be written as a
product of disjoint cycles of total length $2n$. Moreover, if $c =
(a_1, a_2, \ldots, a_k)$ is such a cycle then so is $\bar{c} =
(\bar{a}_1, \bar{a}_2, \ldots, \bar{a}_k)$, and either $c$ and
$\bar{c}$ are disjoint or else $c = \bar{c}$. In the former case,
the product $c \bar{c}$ is said to be a \emph{positive cycle} of
$w$, of length $k$; otherwise $k$ is necessarily even and $c =
\bar{c}$ is said to be a \emph{negative cycle} of $w$, of length
$k/2$. Then the pair of partitions $(\alpha, \beta)$ for which the
parts of $\alpha$ (respectively, $\beta$) are the lengths of the
positive (respectively, negative) cycles of $w$ is a bipartition
of $n$, called the \emph{signed cycle type} of $w$. Two elements
of $B_n$ are congugate if and only if they have the same signed
cycle type. We will denote by $C_{\alpha, \beta}$ the conjugacy
class of elements of $B_n$ of signed cycle type $(\alpha, \beta)$,
by $\chi^{\lambda,\mu}$ the irreducible $B_n$-character
corresponding to $(\lambda, \mu) \vdash n$, and by $\chi(\alpha,
\beta)$ the value of a class function $\chi: B_n \to \CC$ at an
arbitrary element of $C_{\alpha, \beta}$.

For $\lambda = (\lambda_1, \lambda_2,\dots,\lambda_r) \vdash n$,
we set
\[
p^+_\lambda (x, y) \ = \ p^+_{\lambda_1} p^+_{\lambda_2} \cdots
p^+_{\lambda_r} \qquad \text{and} \qquad p^-_\lambda (x, y) \ = \
p^-_{\lambda_1} p^-_{\lambda_2} \cdots p^-_{\lambda_r}
\]
where
\[
p^+_k \ = \ \sum_{i \ge 1} \, (x_i^k + y_i^k) \qquad \text{and}
\qquad p^-_k \ = \ \sum_{i \ge 1} \, (x_i^k - y_i^k)
\]
for $k \ge 1$. Thus $p^+_\lambda (x, y)$ and $p^-_\lambda (x, y)$
are homogeneous elements of degree $n$ of $\Lambda(x) \otimes
\Lambda(y)$. The Frobenius characteristic of a class function
$\chi: B_n \to \CC$ is defined as
 \begin{equation} \label{eq:chB}
   \ch(\chi) \ = \ \frac{1}{2^n n!} \, \sum_{w \in B_n} \chi(w) p_w(x, y)
  \end{equation}
where $p_w(x, y) := p^+_\alpha (x, y) p^-_\beta (x, y)$ if $w \in
B_n$ has signed cycle type $(\alpha, \beta)$. The Frobenius
characteristic is a $\CC$-linear isomorphism from the space of
such class functions on $B_n$ to the degree $n$ homogeneous part
of $\Lambda (x) \otimes \Lambda(y)$. It satisfies
$\ch(\chi^{\lambda, \mu}) = s_\lambda(x) s_\mu(y)$ for every
bipartition $(\lambda, \mu) \vdash n$. The following lemma will be
used in Section~\ref{subsec:derange}.
\begin{lemma} {\rm (\cite[Lemma~21~(i)]{Poi98})}
\label{lem:ch1B} For the trivial character of $B_n$ we have
$\ch(1_{B_n}) = s_n (x)$.
\end{lemma}
The Frobenius formula for $B_n$ states that
  \begin{equation} \label{eq:FrobB}
    p^+_\alpha (x, y) p^-_\beta (x, y) \ = \
    \sum_{(\lambda, \mu) \vdash n} \chi^{\lambda, \mu} (\alpha, \beta) s_\lambda(x) s_\mu(y)
  \end{equation}
for every bipartition $(\alpha,\beta) \vdash n$.

An analogue of Theorem~\ref{thm:roi} for $B_n$ will be proved in
Section~\ref{sec:proof} (see Theorem~\ref{thm:roiB}).

We write $\langle \, \cdot \, , \, \cdot \, \rangle$ for the inner
product on $\Lambda(x)$ (respectively, on $\Lambda (x) \otimes
\Lambda(y)$) for which the Schur functions $s_\lambda(x)$
(respectively, the functions $s_\lambda(x) s_\mu(y)$) form an
orthonormal basis.

\subsection{Quasisymmetric functions}
\label{subsec:quasi}

We denote by $\QSym^n$ the $\CC$-vector space of homogeneous
quasisymmetric functions of degree $n$ in $x$. The
\emph{fundamental quasisymmetric function} associated to $S
\subseteq [n-1]$ is defined as
  \begin{equation} \label{eq:defF(x)}
    F_{n, S} (x) \ =
    \sum_{\substack{i_1 \le i_2 \le \ldots \le i_n \\ j \in S \,\Rightarrow\, i_j < i_{j+1}}}
    x_{i_1} x_{i_2} \cdots x_{i_n}.
  \end{equation}
The set $\{ F_{n, S}(x): S \subseteq [n-1]\}$ is known to be a
basis of $\QSym^n$. The following proposition expresses the Schur
function $s_\lambda(x)$ as a linear combination of the elements of
this basis.
\begin{proposition} {\rm (\cite[Theorem~7.19.7]{StaEC2})}
\label{prop:sFexpansion} For every $\lambda \vdash n$,
  \begin{equation} \label{eq:sFexpansion}
    s_\lambda(x) \ = \ \sum_{Q \in \SYT (\lambda)} F_{n, \Des(Q)} (x).
  \end{equation}
\end{proposition}
Different type $B$ analogues of quasisymmetric functions have been
suggested~\cite{Chow, Petersen05}. The $B_n$-analogues of the
fundamental quasisymmetric functions that we need were introduced
(in the more general setting of $r$-colored permutations) by
Poirier~\cite[Section~3]{Poi98} and were further studied
in~\cite{BH06, HP10}. For a signed set $\sigma = (S, \veps) \in
\Sigma^B (n)$ define
 \begin{equation} \label{eq:defF(xy)}
    F_\sigma (x, y) \ =
    \sum_{\substack{i_1 \le i_2 \le \ldots \le i_n \\ j \in \wDes(\sigma) \,\Rightarrow\, i_j < i_{j+1}}}
    z_{i_1} z_{i_2} \cdots z_{i_n}
  \end{equation}
where $z_{i_j} = x_{i_j}$ if $\teps_j = +$, and $z_{i_j} = y_{i_j}$ 
if $\teps_j = -$ (we use the notation $F_\sigma$, instead of $F_{n, 
\sigma}$, since $n$ is determined by $\sigma = (S, \varepsilon)$ as the
largest element of $S$). For example, if $n=6$ and $S = \{2, 3, 5,
6\}$ with sign vector $\teps = (+, +, -, -, -, +)$, then
\[
    F_\sigma (x, y) \ = \sum_{1 \le i_1 \le i_2 < i_3 < i_4 \le i_5
     \le i_6} x_{i_1} x_{i_2} y_{i_3} y_{i_4} y_{i_5} x_{i_6}.
\]

We leave it to the reader to verify that this definition is
equivalent to the one in~\cite{Poi98}. The functions $F_\sigma (x,
y)$, where $\sigma$ ranges over all signed sets in $\Sigma^B (n)$,
are linearly independent over $\CC$ \cite[Corollary~9]{Poi98} and,
in fact, form a $\CC$-basis of a natural signed analogue of
$\QSym^n$; see \cite[Section~3.2]{HP10}.

\section{Main results: Character formulas and fine sets}
\label{sec:main}

This section introduces the concept of fine set for the
hyperoctahedral group and explains the main results of this paper,
to be proved in the following section. The corresponding concepts
and results from \cite{AR15} for the symmetric group will first be
reviewed.

\subsection{Review of results for the symmetric group}
\label{subsec:main_A} The following key definition from
\cite{AR15} uses the notion of unimodality with respect to a
composition, defined in Section~\ref{subsec:char}.

\begin{definition} \label{def:fineA} {\rm (\cite[Definition~1.3]{AR15})}
Let $\chi$ be a character of the symmetric group $\mathfrak{S}_n$.
A set $\bB$, endowed with a map $\Des: \bB \to 2^{[n-1]}$, is said
to be a \emph{fine set} for $\chi$ if
  \begin{equation} \label{eq:fineA}
    \chi(\alpha) \ = \ \sum_{b \in \bB_\alpha} (-1)^{|\Des(b) \sm
    S(\alpha)|}
  \end{equation}
for every composition $\alpha$ of $n$, where $\bB_\alpha$ is the
set of all elements $b \in \bB$ for which $\Des(b)$ is
$\alpha$-unimodal.
\end{definition}

For example, according to Theorem~\ref{thm:roi}, the set
$\SYT(\lambda)$, endowed with the usual descent map $\Des:
\SYT(\lambda) \to 2^{[n-1]}$, is a fine set for the irreducible
character $\chi^\lambda$. Several other examples (including some
new ones) appear in Sections~\ref{sec:Knuth}--\ref{sec:arc}.

The following theorem is a version of the main result of
\cite{AR15}; we include a proof since this version is only
implicit in \cite{AR15}.

\begin{theorem} \label{thm:mainA}
Let $\chi$ be a character of the symmetric group $\mathfrak{S}_n$
and $\bB$ be a set endowed with a map $\Des: \bB \to 2^{[n-1]}$.
Then $\bB$ is a fine set for $\chi$ if and only if
  \begin{equation} \label{eq:mainA}
    \ch(\chi) \ = \ \sum_{b \in \bB} F_{n, \Des(b)} (x).
  \end{equation}
In particular, the distribution of the descent set over $\bB$ is
uniquely determined by $\chi$.
\end{theorem}
\begin{proof}
Let us express $\chi = \sum_{\lambda \vdash n} c_\lambda
\chi^\lambda$ as a linear combination of irreducible characters.
Using Proposition~\ref{prop:sFexpansion} we find that
\[ \ch(\chi) \ = \ \sum_{\lambda \vdash n} c_\lambda \ch
   (\chi^\lambda) \ = \ \sum_{\lambda \vdash n} c_\lambda
   s_\lambda(x) \ = \ \sum_{\lambda \vdash n} c_\lambda \sum_{Q
   \in \SYT(\lambda)} F_{n, \Des(Q)} (x). \]
From this we conclude that (\ref{eq:mainA}) holds if and only if
  \begin{equation} \label{eq:mainA2}
    |\{b \in \bB: \Des(b) = S\}| \ = \ \sum_{\lambda \vdash n}
    c_\lambda \cdot |\{Q \in \SYT(\lambda): \Des(Q) = S \}|
  \end{equation}
for every $S \subseteq [n-1]$. Assuming (\ref{eq:mainA}) holds and
using (\ref{eq:mainA2}) and Theorem~\ref{thm:roi}, we get

\begin{eqnarray*}
  \chi (\alpha) &=& \sum_{\lambda \vdash n} c_\lambda
  \cdot \chi^\lambda (\alpha) \ = \ \sum_{\lambda \vdash n} c_\lambda
  \sum_{\substack{Q \in \SYT(\lambda) \\ \Des(Q) \in U_\alpha}}
  (-1)^{|\Des(Q) \sm S(\alpha)|} \\
  &=& \sum_{\substack{S \subseteq [n-1] \\ S \in U_\alpha}}
  (-1)^{|S \sm S(\alpha)|} \ \sum_{\lambda \vdash n} c_\lambda \cdot
  |\{Q \in \SYT(\lambda): \Des(Q) = S \}| \\
  & & \\
  &=& \sum_{\substack{S \subseteq [n-1] \\ S \in U_\alpha}}
  (-1)^{|S \sm S(\alpha)|} \cdot |\{b \in \bB: \Des(b) = S \}| \\
  & & \\
  &=& \sum_{b \in \bB_\alpha} (-1)^{|\Des(b) \sm S(\alpha)|}
\end{eqnarray*}
for every $\alpha \models n$, and conclude that $\bB$ is
a fine set for $\chi$.

Conversely, assume that $\bB$ is a fine set for $\chi$ and let
$\bB'$ be the disjoint union, taken over all partitions $\lambda
\vdash n$, of $c_\lambda$ copies of $\SYT(\lambda)$, endowed with
the usual descent map $\Des: \bB' \to 2^{[n-1]}$. By
Theorem~\ref{thm:roi} the set $\bB'$ is also fine for $\chi$.
Therefore, \cite[Theorem~1.4]{AR15} implies that the descent set
is equidistributed over $\bB$ and $\bB'$. This exactly means that
(\ref{eq:mainA2}) holds for every $S \subseteq [n-1]$ and the
proof follows.
\end{proof}

\begin{corollary} \label{cor:mainA} {\rm (see \cite[Theorem~1.5]{AR15})}
Let $\bB$ be a set endowed with a map $\Des: \bB \to 2^{[n-1]}$.
The following are equivalent:
  \begin{enumerate}[(i)]
    \item $\bB$ is a fine set for some (non-virtual) character
               of $\mathfrak{S}_n$.
    \item The quasisymmetric function $F_\bB(x) := \sum_{b \in
                \bB} F_{n, \Des(b)} (x)$ is symmetric and
                Schur-positive.
    \item There exist nonnegative integers $a_\lambda$, for
                 $\lambda \vdash n$, such that
      \begin{equation} \label{eq:desdistrA}
        \sum_{b \in \bB} z^{\Des(b)} \ = \sum_{\lambda \vdash n}
        a_\lambda \sum_{Q \in \SYT(\lambda)} z^{\Des(Q)}
      \end{equation}
     \end{enumerate}
where $z = (z_1,\dots,z_{n-1})$. Moreover, if these conditions
hold, then $a_\lambda = \langle F_\bB(x), s_\lambda(x) \rangle$
for each $\lambda \vdash n$.
\end{corollary}

\subsection{Results for the hyperoctahedral group}
\label{subsec:main_B} To state the $B_n$-analogues of these
results, we will use the language of signed compositions and sets,
explained in Section~\ref{sec:back}. Let $\sigma = (S, \veps) \in
\Sigma^B(n)$ be a signed set, $\gamma \in \Comp^B(n)$ a signed
composition of $n$ with set of partial sums $S(\gamma) = \{r_1,
r_2,\dots,r_k\}$ (so $r_k = n$), and set $r_0 = 0$. The signed set
$\sigma$ is called \emph{$\gamma$-unimodal} if $S$ is unimodal
with respect to the unsigned composition of $n$ corresponding to
$\gamma$.
\begin{definition} \label{def:weightB}
The \emph{weight} $\wei_\gamma(\sigma)$ of $\sigma = (S, \veps)
\in \Sigma^B(n)$ (equivalently, of the signed composition of $n$
corresponding to $\sigma$) with respect to $\gamma$ is defined as
follows:
  \begin{itemize}
    \itemsep=0pt
    \item[$\bullet$]
    $\wei_\gamma(\sigma) := 0$ 
    if either $\sigma$ is not $\gamma$-unimodal, or else
    for some index $1 \le i \le k$ the sign vector $\teps \in \{-, +\}^n$ of $\sigma$
    is not constant on the set $\{r_{i-1} + 1,\dots,r_i\}$.
    \item[$\bullet$]
    Otherwise we set
      \begin{equation} \label{eq:defwei}
        \wei_\gamma (\sigma) \ := \
        (-1)^{|S \sm S(\gamma)| \, + \, n_\gamma(\sigma)}
      \end{equation}
    where $n_\gamma (\sigma)$ is the number of indices $i$ for which
    the elements of $\{r_{i-1} + 1,\dots,r_i\}$ are assigned the negative sign by
    (the sign vectors of) both $\sigma$ and $\gamma$.
  \end{itemize}
\end{definition}

Given a character $\chi$ of $B_n$ and a signed composition $\gamma
\in \Comp^B(n)$, we will write $\chi(\gamma)$ for the value of
$\chi$ at the elements of $B_n$ of signed cycle type $(\alpha,
\beta) \vdash n$, where $\alpha$ and $\beta$ are the partitions
obtained by reordering the unbarred and barred parts of $\gamma$,
respectively, in weakly decreasing order.
\begin{definition} \label{def:fineB}
Let $\chi$ be a character of the hyperoctahedral group $B_n$. A
set $\bB$, endowed with a map $\cDes: \bB \to \Sigma^B(n)$, is
said to be a \emph{fine set} for $\chi$ if
  \begin{equation} \label{eq:fineB}
    \chi(\gamma) \ = \ \sum_{b \in \bB} \wei_\gamma (\cDes(b))
  \end{equation}
for every signed composition $\gamma$ of $n$.
\end{definition}

The main results of this paper are as follows.

\begin{theorem} \label{thm:main}
Let $\chi$ be a character of the hyperoctahedral group $B_n$ and
$\bB$ a set endowed with a map $\cDes: \bB \to \Sigma^B(n)$. Then
$\bB$ is a fine set for $\chi$ if and only if
  \begin{equation} \label{eq:main}
    \ch(\chi) \ = \ \sum_{b \in \bB} F_{\cDes(b)} (x, y).
  \end{equation}
In particular, the distribution of the signed descent set over
$\bB$ is uniquely determined by $\chi$.
\end{theorem}

To state the third condition of the following corollary, given
$\sigma = (S, \veps) \in \Sigma^B(n)$ we write $z^\sigma =
\prod_{i \in S} z_i$, where $z_i = x_i$ (respectively, $z_i =
y_i$) if $\veps_i = +$ (respectively, $\veps_i = -$).

\begin{corollary} \label{cor:main}
Let $\bB$ be a set endowed with a map $\cDes: \bB \to
\Sigma^B(n)$. The following are equivalent:
  \begin{enumerate}[(i)]
     \item
        $\bB$ is a fine set for some (non-virtual) character of $B_n$.
    \item
        The quasisymmetric function $F_\bB(x, y) := \sum_{b \in \bB} F_{\cDes(b)} (x,y)$
        is a Schur-positive element of $\Lambda(x) \otimes \Lambda(y)$.
    \item
        There exist nonnegative integers $a_{\lambda, \mu}$, for $(\lambda, \mu) \vdash n$,
        such that
      \begin{equation} \label{eq:desdistrB}
        \sum_{b \in \bB} z^{\cDes(b)} \ = \
        \sum_{(\lambda, \mu) \vdash n} a_{\lambda, \mu}
        \sum_{Q \in \SYT(\lambda, \mu)} z^{\cDes(Q)}.
      \end{equation}
  \end{enumerate}
Moreover, if these conditions hold, then $a_{\lambda, \mu} =
\langle F_\bB(x, y), s_\lambda(x) s_\mu(y) \rangle$ for each
$(\lambda, \mu) \vdash n$.
\end{corollary}

\section{Proofs}
\label{sec:proof}

This section proves Theorem~\ref{thm:main} and deduces
Corollary~\ref{cor:main} from it. The proof of one direction of
Theorem~\ref{thm:main} follows that of Theorem~\ref{thm:mainA}, as
described in Section~\ref{sec:main}; it is based on a formula for
the values of the irreducible characters of $B_n$ (see
Theorem~\ref{thm:roiB}) which is analogous to Roichman's rule for
the irreducible characters of $\mathfrak{S}_n$
(Theorem~\ref{thm:roi}). The challenge of finding analogues of
Roichman's rule for other (complex) reflection groups was
implicitly raised in~\cite[page~45]{BR99}. The proof of the other
direction of Theorem~\ref{thm:main} requires the invertibility of
a suitable weight matrix.

\subsection{Irreducible characters and bitableaux}
\label{subsec:irr_char_tab} Given a composition $\alpha$, denote
by $\alpha^+$ (respectively, $\alpha^-$) the signed composition
obtained from $\alpha$ by considering its parts as unbarred
(respectively, barred). The following result states that the set
of standard Young bitableaux of shape $(\lambda, \mu) \vdash n$,
endowed with the standard colored descent map defined in
Section~\ref{subsec:perm}, is a fine set for the irreducible
chararcter $\chi^{\lambda, \mu}$ of $B_n$.

\begin{theorem} \label{thm:roiB}
For all bipartitions $(\lambda, \mu) \vdash n$ and $(\alpha,
\beta) \vdash n$ and for every permutation $\gamma$ of the parts
of $\alpha^+$ and $\beta^-$,
  \begin{equation} \label{eq:chiroiB}
    \chi^{\lambda, \mu} (\alpha, \beta) \ = \sum_{Q \in \SYT(\lambda,
    \mu)} \wei_\gamma (\cDes(Q)).
  \end{equation}
In other words, $\SYT(\lambda, \mu)$ (with the associated map
$\cDes$) is a fine set for the irreducible character
$\chi^{\lambda, \mu}$.
\end{theorem}
\begin{proof}
To compute $\chi^{\lambda, \mu} (\alpha, \beta)$ we expand the
left-hand side of (\ref{eq:FrobB}) as

\begin{eqnarray*}
  p^+_\alpha (x, y) p^-_\beta (x, y) &=& \prod_{i=1}^{\ell(\alpha)}
  \, (p_{\alpha_i}(x) + p_{\alpha_i}(y)) \, \prod_{j=1}^{\ell(\beta)}
  \, (p_{\beta_j}(x) - p_{\beta_j}(y)) \\
  & & \\
  &=& \sum_{\varepsilon, \zeta} \,
  (-1)^{|\{j: \, \zeta_j = -\}|} \, p_\nu (x) p_\xi (y)
\end{eqnarray*}
where the sum ranges over all the vectors $\veps = (\veps_1,
\ldots, \veps_{\ell(\alpha)}) \in \{-, +\}^{\ell(\alpha)}$ and
$\zeta = (\zeta_1, \ldots, \zeta_{\ell(\beta)}) \in \{-,
+\}^{\ell(\beta)}$, and where $\nu = \nu^{\veps, \zeta}$
(respectively, $\xi = \xi^{\veps, \zeta}$) is the composition
consisting of the parts $\alpha_i$ of $\alpha$ with $\veps_i = +$
(respectively, $\veps_i = -$) followed by the parts $\beta_j$ of
$\beta$ with $\zeta_j = +$ (respectively, $\zeta_j = -$).
Expressing each of $p_\nu (x)$ and $p_\xi (y)$ in the basis of
Schur functions, according to (\ref{eq:FrobA}), and comparing to
(\ref{eq:FrobB}) we get

 \begin{equation} \label{eq:chiroiB1}
   \chi^{\lambda, \mu} (\alpha, \beta) \ = \
   \sum_{\veps, \zeta} \, (-1)^{|\{j: \, \zeta_j = -\}|} \,
   \chi^\lambda (\nu^{\veps, \zeta}) \chi^\mu (\xi^{\veps, \zeta})
 \end{equation}
where the summation and the compositions $\nu^{\veps, \zeta}$ and
$\xi^{\veps, \zeta}$ determined by the sign vectors $\veps$ and
$\zeta$ are as before.

We now derive a similar formula for the right-hand side of
(\ref{eq:chiroiB}), denoted by $\psi^{\lambda, \mu} (\gamma)$.
Given sign vectors $\veps = (\veps_1, \ldots,
\veps_{\ell(\alpha)}) \in \{-, +\}^{\ell(\alpha)}$ and $\zeta =
(\zeta_1, \ldots, \zeta_{\ell(\beta)}) \in \{-,
+\}^{\ell(\beta)}$, denote by $\gamma^{\veps, \zeta}$ the signed
composition whose underlying composition is that of $\gamma$ and
whose parts are unbarred or barred, according to whether the
corresponding parts of $\alpha$ and $\beta$ are assigned the $+$
or $-$ sign by $\veps$ and $\zeta$, respectively. Denote by
$\gamma_+^{\veps, \zeta}$ (respectively, $\gamma_-^{\veps,
\zeta}$) the composition obtained from $\gamma^{\veps, \zeta}$ by
removing the barred (respectively, unbarred) parts and forgetting
the bars. Express the set $[n]$ as the disjoint union of
contiguous segments whose cardinalities are the parts of $\gamma$,
and denote by $R_+^{\veps, \zeta}$ (respectively, $R_-^{\veps,
\zeta}$) the union of those segments which correspond to the
unbarred (respectively, barred) parts of $\gamma^{\veps, \zeta}$.
The definitions of weight and signed descent set show that
\begin{multline*}
\psi^{\lambda, \mu} (\gamma) \\ =  \sum_{\veps, \zeta} \,
(-1)^{|\{j: \, \zeta_j = -\}|} \left( \, \sum_{Q^+} \,
(-1)^{|\Des(Q^+) \sm S(\gamma_+^{\veps, \zeta})|} \right) \left(
\, \sum_{Q^-} \, (-1)^{|\Des(Q^-) \sm S(\gamma_-^{\veps, \zeta})|}
\right)
\end{multline*}
where the outer sum ranges over all sign vectors $\veps,\zeta$ as
above and the inner sums range over all Young tableaux $Q^+$ of
shape $\lambda$ with content $R_+^{\veps, \zeta}$ and
$\gamma_+^{\veps, \zeta}$-unimodal descent set and over all Young
tableaux $Q^-$ of shape $\mu$ with content $R_-^{\veps, \zeta}$
and $\gamma_-^{\veps, \zeta}$-unimodal descent set.
Theorem~\ref{thm:roi} and the previous formula imply that

\begin{equation} \label{eq:chiroiB2}
\psi^{\lambda, \mu} (\gamma) \ = \ \sum_{\veps, \zeta} \,
(-1)^{|\{j: \, \zeta_j = -\}|} \, \chi^\lambda(\gamma_+^{\veps,
\zeta}) \chi^\mu(\gamma_-^{\veps, \zeta})
\end{equation}
where the sum ranges over all sign vectors $\veps \in \{-,
+\}^{\ell(\alpha)}$ and $\zeta \in \{-, +\}^{\ell(\beta)}$. Since
the values $\chi(\nu)$ of an irreducible character $\chi$ of a
symmetric group do not depend on the ordering of the parts of the
composition $\nu$, Equations (\ref{eq:chiroiB1}) and
(\ref{eq:chiroiB2}) imply that $\chi^{\lambda, \mu} (\alpha,
\beta) = \psi^{\lambda, \mu} (\gamma)$, as claimed by the theorem.
\end{proof}

\begin{proposition} \label{prop:schurprod}
For all partitions $\lambda, \mu$
\[
    s_\lambda (x) s_\mu (y) \ = \sum_{Q \in \SYT(\lambda, \mu)}
    F_{\cDes(Q)} (x, y).
\]
\end{proposition}
\begin{proof}
This statement follows from \cite[Corollary~8]{HP10} and
Proposition~\ref{prop:sFexpansion}. For a direct proof, it
suffices to describe a bijection from the set of all pairs of
semistandard Young tableaux $S$ and $T$ of shape $\lambda$ and
$\mu$, respectively, to the set of all pairs $(Q, u)$ of standard
Young bitableau $Q \in \SYT(\lambda, \mu)$ and monomials $u$ which
appear in the expansion of $F_{\cDes(Q)} (x, y)$, such that if
$(Q, u)$ corresponds to $(S, T)$ then $x^S y^T = u$. Given $(S,
T)$, we define $Q$ by numbering the entries of $T$ equal to 1,
read from bottom to top and from left to right, with the first few
positive integers $1, 2,\dots$; then the entries of $S$ equal to
1, read in the same fashion, with the next few positive integers;
then the entries of $T$ equal to 2 and so on, and set $u = x^S
y^T$. We leave to the interested reader to verify that this
induces a well defined map which has the claimed properties.
\end{proof}

The following statement furnishes one direction of
Theorem~\ref{thm:main}.
\begin{proposition} \label{prop:main1}
Let $\chi$ be a character of the group $B_n$ and let $\bB$ be a
set endowed with a map $\cDes: \bB \to \Sigma^B(n)$. If
  \begin{equation} \label{eq:main1}
    \ch(\chi) \ = \ \sum_{b \in \bB} F_{\cDes(b)} (x, y),
  \end{equation}
then $\bB$ is a fine set for $\chi$.
\end{proposition}
\begin{proof}
Let us express $\chi = \sum_{(\lambda, \mu) \vdash n} c_{\lambda,
\mu} \chi^{\lambda, \mu}$ as a linear combination of irreducible
characters. We then have, by Proposition~\ref{prop:schurprod},
\begin{eqnarray*}
   \ch(\chi) & = & \sum_{(\lambda, \mu) \vdash n} c_{\lambda, \mu} \,
   \ch (\chi^{\lambda, \mu}) \ = \ \sum_{(\lambda, \mu) \vdash n}
  c_{\lambda, \mu} \, s_\lambda(x) s_\mu(y) \\ 
  &=& \sum_{(\lambda, \mu)
  \vdash n} c_{\lambda, \mu} \sum_{Q \in \SYT(\lambda, \mu)}
  F_{\cDes(Q)} (x, y).
\end{eqnarray*}
Comparing with (\ref{eq:main1}) we get
\[
    |\{b \in \bB: \cDes(b) = \sigma \}| \ = \sum_{(\lambda, \mu)
    \vdash n} c_{\lambda, \mu} \cdot |\{Q \in \SYT(\lambda, \mu):
    \cDes(Q) = \sigma \}|
\]
for every $\sigma \in \Sigma^B(n)$. Using this equation and
Theorem~\ref{thm:roiB}, we get

\begin{eqnarray*}
  \chi (\alpha, \beta) &=& \sum_{(\lambda, \mu) \vdash n} c_{\lambda,
  \mu} \, \chi^{\lambda, \mu} (\alpha, \beta) \ = \ \sum_{(\lambda,
  \mu) \vdash n} c_{\lambda, \mu} \sum_{Q \in \SYT(\lambda,
  \mu)} \wei_\gamma (\cDes(Q)) \\
  &=& \sum_{\sigma  \in \Sigma^B(n)} \wei_\gamma (\sigma) \,
  \sum_{(\lambda, \mu) \vdash n} c_{\lambda, \mu} \cdot
  |\{Q \in \SYT(\lambda, \mu): \cDes(Q) = \sigma \}| \\
  &=& \sum_{\sigma  \in \Sigma^B(n)} \wei_\gamma (\sigma) \cdot
  |\{b \in \bB: \cDes(b) = \sigma \}| \\
  &=& \sum_{b \in \bB} \wei_\gamma (\cDes(b))
\end{eqnarray*}
for every bipartition $(\alpha,\beta) \vdash n$ and
every permutation $\gamma$ of the parts of $\alpha^+$ and
$\beta^-$ and the proof follows.
\end{proof}

For the proof of the other direction of Theorem~\ref{thm:main} we
need the following $B_n$-analogue of \cite[Theorem~1.4]{AR15},
which will be proved in Section~\ref{subsec:weight_matrix}.
\begin{theorem}\label{thm:char_to_cDes}
If $\bB$ is a fine set for a $B_n$-character $\chi$, then the
distribution of $\cDes$ over $\bB$ is uniquely determined by
$\chi$.
\end{theorem}

\begin{proof}[Proof of Theorem~\ref{thm:main}] We only need to
prove the converse of Proposition~\ref{prop:main1}. Let $\bB$ be a
set which, endowed with a map $\cDes: \bB \to \Sigma^B(n)$, is
fine for a $B_n$-character $\chi$. We have to show that
(\ref{eq:main1}) holds. Express $\chi = \sum_{(\lambda, \mu)
\vdash n} c_{\lambda, \mu} \chi^{\lambda, \mu}$ as a linear
combination of irreducible characters. Let $\bB'$ be the disjoint
union, taken over all bipartitions $(\lambda, \mu) \vdash n$, of
$c_{\lambda, \mu}$ copies of $\SYT(\lambda, \mu)$, endowed with
the corresponding descent map $\cDes$. By Theorem~\ref{thm:roiB}
the set $\bB'$ is also fine for $\chi$.
Theorem~\ref{thm:char_to_cDes} now implies that $\cDes$ is
equidistributed over $\bB$ and $\bB'$. As shown at the beginning
of the proof of Proposition~\ref{prop:main1},
\[
\ch(\chi) \ = \ \sum_{b \in \bB'} F_{\cDes(b)} (x, y).
\]
Therefore, (\ref{eq:main1}) holds as well. 
\end{proof}

Finally, we deduce Corollary~\ref{cor:main} from
Theorem~\ref{thm:main}.

\begin{proof}[Proof of Corollary~\ref{cor:main}] The
equivalence (i) $\Leftrightarrow$ (ii) is a direct consequence of
Theorem~\ref{thm:main}. The equivalence (ii) $\Leftrightarrow$
(iii) follows from Proposition~\ref{prop:schurprod} and the linear
independence of the functions $F_\sigma (x, y)$ for $\sigma \in
\Sigma^B(n)$. \end{proof}

\subsection{The weight matrix}
\label{subsec:weight_matrix}

We now prove Theorem~\ref{thm:char_to_cDes}, the missing step in
the proof of Theorem~\ref{thm:main}, using a $B_n$-analogue of the
Hadamard-type matrix $A_n$ from~\cite{AR15, A14}.

By Definition~\ref{def:weightB}, a weight $\wei_\gamma(\sigma)$ is
defined for any signed composition $\gamma \in \Comp^B(n)$ and
signed set $\sigma \in \Sigma^B(n)$. The size of each of the sets
$\Comp^B(n)$ and $\Sigma^B(n)$ is $d_n := 2 \cdot 3^{n-1}$. Fixing
a linear order (to be specified later) on each of these sets, the
weights can be arranged into a \emph{weight matrix}
\[
A_n \, = \, (\wei_\gamma(\sigma))_{\gamma \in \Comp^B(n),\ \sigma
\in \Sigma^B(n)} \in \{1,0,-1\}^{d_n \times d_n}.
\]
For a set $\bB$, endowed with a map $\cDes: \bB \to \Sigma^B(n)$,
let $v = v(\bB)$ be the column vector of length $d_n$ with entries
\[
v_\sigma \, = \, |\{b \in \bB: \, \cDes(b) = \sigma\}| \qquad
(\forall \sigma \in \Sigma^B(n)),
\]
where $\Sigma^B(n)$ is in the specified linear order. For a
character $\chi$ of the hyperoctahedral group $B_n$, let $u =
u(\chi)$ be the column vector of length $d_n$ with entries
\[
u_\gamma = \chi(\gamma) \qquad (\forall \gamma \in \Comp^B(n)),
\]
where $\Comp^B(n)$ is in the specified order.
Definition~\ref{def:fineB} can now be restated as follows:
\[
\text{$\bB$ is a fine set for $\chi$} \,\Longleftrightarrow\,
u(\chi) = A_n \cdot v(\bB).
\]
Theorem~\ref{thm:char_to_cDes} is thus implied by the following
result.

\begin{theorem} \label{thm:det_A}
The matrix $A_n$ is invertible. In fact,
\[
|\det(A_n)| \ = \prod_{\gamma \in \Comp^B(n)} m_\gamma
\]
where, for a signed composition $\gamma \in \Comp^B(n)$ with part
sizes $\gamma_1, \ldots, \gamma_k$,
\[
m_\gamma \ := \ 2^{k/2} \cdot \prod_{i=1}^{k} \gamma_i.
\]
\end{theorem}

Theorem~\ref{thm:det_A} will be proved using a recursive formula
for $A_n$ and for an auxiliary matrix $\hA_n$.

Consider now the alphabet $L = \{0,1,*\}$, the set $L^n$ of all
words of length $n$ with letters from $L$, and the set $L_*^n$ of
all words $w \in L^n$ in which the last letter is \emph{not} $*$.
Any signed set $\sigma = (S, \veps) \in \Sigma^B(n)$ is uniquely
represented by a word $w(\sigma) = a_1 \cdots a_n \in L_*^n$,
where
\[
a_i \, = \, \begin{cases}
0, & \text{if } i \in S \text{ and } \veps(i) = +; \\
1, & \text{if } i \in S \text{ and } \veps(i) = -; \\
*, & \text{if } i \not\in S.
\end{cases}
 \]
Recalling that we always have $n \in S$, the map $\sigma \mapsto
w(\sigma)$ is well defined and is clearly a bijection from
$\Sigma^B(n)$ to $L_*^n$. The natural bijection from $\Comp^B(n)$
to $\Sigma^B(n)$ (see Section~\ref{subsec:comp}) yields a
corresponding representation of signed compositions by words in
$L_*^n$.

By the above, we can consider the rows and columns of the weight
matrix $A_n$ to be indexed by words $w \in L_*^n$. Let $\hA_n$ be
the matrix obtained from $A_n$ by setting to zero all the entries
indexed by $(w,w')$ for which the initial sequence of $*$-s in $w$
is longer than the initial sequence of $*$-s in $w'$, namely:
$(\hA_n)_{w, w'} = 0$ if $i(w) > i(w')$, where
\[
i(w) := \min \{i \,|\, w_i \ne *\}.
\]
This corresponds to a pair $(\gamma, \sigma)$ for which there is
an element of $S(\sigma)$ before the end of the first part of the
signed composition $\gamma$.

For a word $w \in L_*^n$, let $c(w)$ be the first letter in $w$
which is not $*$:
\[
c(w) := w_{i(w)}.
\]
For a letter $a' \in \{0, 1\}$, let $A_n^{a'}$ be the matrix
obtained from $A_n$ by setting to zero all the entries in columns
indexed by words $w'$ with $c(w') \ne a'$. Clearly, $A_n = A_n^0 +
A_n^1$. We use similar notation for $\hA_n$.

Definition~\ref{def:weightB} of the weight $\wei_\gamma(\sigma)$
implies the following recursive properties of $A_n$ and $\hA_n$:
\begin{lemma}
For $n \ge 1$, letters $a, a' \in L$ and words $w, w' \in L_*^n$,
\[
(A_{n+1})_{aw, a'w'} \ = \
\begin{cases}
(-1)^{a \cdot a'} (A_n)_{w, w'}, & \text{if } a, a' \in \{0, 1\}; \\
(-1)^{a \cdot c(w')} (A_n)_{w, w'}, & \text{if } a \in \{0, 1\},\, a' = *; \\
-(A_n^{a'})_{w, w'}, & \text{if } a = *,\, a' \in \{0, 1\}; \\
(\hA_n)_{w, w'}, & \text{if } a = a' = *
\end{cases}
\]
and
\[
(\hA_{n+1})_{aw, a'w'} \ = \
\begin{cases}
(-1)^{a \cdot a'} (A_n)_{w, w'}, & \text{if } a, a' \in \{0, 1\}; \\
(-1)^{a \cdot c(w')} (A_n)_{w, w'}, & \text{if } a \in \{0, 1\},\, a' = *; \\
0, & \text{if } a = *,\, a' \in \{0, 1\}; \\
(\hA_n)_{w, w'}, & \text{if } a = a' = *.
\end{cases}
\]
Also,
\[
(A_1)_{a, a'} \, = \, (\hA_1)_{a, a'} \, = \, (-1)^{a \cdot a'}
\qquad (\forall a, a' \in \{0, 1\}).
\]
\end{lemma}

Specifying a linear order on $L_*^n$, we can write these
recursions in matrix form.
\begin{corollary}
Using the linear order $0 < 1 < *$ on $L$ and the resulting
lexicographic order on $L_*^n$,
\[
A_{n+1} \, = \,
\begin{pmatrix}
A_n & A_n & A_n \\
A_n & -A_n & A_n^0 - A_n^1 \\
-A_n^0 & -A_n^1 & \hA_n
\end{pmatrix}
\qquad (n \ge 1),
\]

\[
\hA_{n+1} \, = \,
\begin{pmatrix}
A_n & A_n & A_n \\
A_n & -A_n & A_n^0 - A_n^1 \\
0 & 0 & \hA_n
\end{pmatrix}
\qquad (n \ge 1)
\]
and
\[
A_1 \, = \, \hA_1 \, = \,
\begin{pmatrix}
1 & 1 \\
1 & -1
\end{pmatrix}.
\]
\end{corollary}

\begin{corollary} \label{cor:det_recursion}
For $n \ge 1$,
\[
\det (A_{n+1}) \, = \, \det (A_n) \cdot \det (-2A_n) \cdot \det
(A_n + \hA_n)
\]
and
\[
\det (\hA_{n+1}) \, = \, \det (A_n) \cdot \det (-2A_n) \cdot \det
(\hA_n),
\]
with
\[
\det (A_1) \, = \, \det (\hA_1) \, = \, -2.
\]
More generally, for every real $\alpha$,
\[
\det (\alpha A_{n+1} + (1 - \alpha) \hA_{n+1}) \ = \ \det (A_n)
\cdot \det (-2A_n) \cdot \det (\alpha A_n + \hA_n)
\]
for $n \ge 1$, with
\[
\det (\alpha A_1 + (1 - \alpha) \hA_1) \, = \, -2.
\]
\end{corollary}
\begin{proof}
Using the fact that $A_n = A_n^0 + A_n^1$ (and similarly for
$\hA_n$), we can write

\[
A_{n+1} \, = \,
\begin{pmatrix}
A_n^0 + A_n^1 & A_n^0 + A_n^1 & A_n^0 + A_n^1 \\
A_n^0 + A_n^1 & - A_n^0 - A_n^1 & A_n^0 - A_n^1 \\
- A_n^0 & - A_n^1 & \hA_n^0 + \hA_n^1
\end{pmatrix}
\qquad (n \ge 1)
\]
and, more generally,
\[
\alpha A_{n+1} + (1 - \alpha) \hA_{n+1} \ = \
\begin{pmatrix}
A_n^0 + A_n^1 & A_n^0 + A_n^1 & A_n^0 + A_n^1 \\
A_n^0 + A_n^1 & - A_n^0 - A_n^1 & A_n^0 - A_n^1 \\
- \alpha A_n^0 & - \alpha A_n^1 & \hA_n^0 + \hA_n^1
\end{pmatrix}
\qquad (n \ge 1),
\]
with
\[
\alpha A_1 + (1 - \alpha) \hA_1 \ = \
\begin{pmatrix}
1 & 1 \\
1 & -1
\end{pmatrix}.
\]
The set of nonzero columns of $A_n^0$ is disjoint from the
corresponding set for $A_n^1$. We can thus perform elementary
column operations using, separately, the columns of $A_n^0$ and
$A_n^1$, to get
\begin{multline*}
\det (\alpha A_{n+1} + (1 - \alpha) \hA_{n+1}) = \det
\begin{pmatrix}
A_n^0 + A_n^1 & A_n^0 + A_n^1 & A_n^0 + A_n^1 \\
A_n^0 + A_n^1 & - A_n^0 - A_n^1 & A_n^0 - A_n^1 \\
- \alpha A_n^0 & - \alpha A_n^1 & \hA_n^0 + \hA_n^1
\end{pmatrix} \\
= \det
\begin{pmatrix}
A_n^0 + A_n^1 & 0 & 0 \\
A_n^0 + A_n^1 & - 2A_n^0 - 2A_n^1 & - 2A_n^1 \\
- \alpha A_n^0 & \alpha A_n^0 - \alpha A_n^1 & \alpha A_n^0 +
\hA_n^0 + \hA_n^1
\end{pmatrix} \\ 
= \det
\begin{pmatrix}
A_n^0 + A_n^1 & 0 & 0 \\
A_n^0 + A_n^1 & - 2A_n^0 - 2A_n^1 & 0 \\
- \alpha A_n^0 & \alpha A_n^0 - \alpha A_n^1 & \alpha A_n^0 +
\alpha A_n^1 + \hA_n^0 + \hA_n^1
\end{pmatrix}.
\end{multline*}
Note that in the last step we used only half of the
columns of the middle $-2A_n^0 - 2A_n^1$ to annihilate the columns
of $-2A_n^1$. The block triangular structure of the resulting
matrix implies that
\[
\det (\alpha A_{n+1} + (1 - \alpha) \hA_{n+1}) \ = \ \det (A_n)
\cdot \det (-2A_n) \cdot \det (\alpha A_n + \hA_n)
\]
for $n \ge 1$, with
\[
\det (\alpha A_1 + (1 - \alpha) \hA_1) \, = \, -2.
\]
The special cases $\alpha = 1$ and $\alpha = 0$ give the results
for $A_n$ and $\hA_n$.
\end{proof}

The following statement implies Theorem~\ref{thm:det_A} and
completes the proof of Theorem~\ref{thm:main}.

\begin{corollary}\label{cor:det_alpha_A}
For all real numbers $\alpha$ and positive integers $n$,
\begin{equation}\label{eq:det_alpha}
\det (\alpha A_n + (1 - \alpha) \hA_n) \ = \ - \prod_{\gamma \in
\Comp^B(n)} m_\gamma(\alpha)
\end{equation}
where, for a signed composition $\gamma \in \Comp^B(n)$ with part
sizes $\gamma_1, \ldots, \gamma_k$,
\[
m_{\gamma}(\alpha) \ = \ 2^{k/2} \cdot (\alpha \gamma_1 + (1 -
\alpha)) \cdot \prod_{i = 2}^{k} \gamma_i.
\]
\end{corollary}
\begin{proof}
Setting
\[
\Delta_n(\alpha) \ := \ \det (\alpha A_n + (1 - \alpha) \hA_n),
\]
Corollary~\ref{cor:det_recursion} implies that, for $\alpha \ne
-1$,
\[
\Delta_{n+1}(\alpha) \ = \ (2(\alpha + 1))^{d_n} \cdot
\Delta_n(1)^2 \cdot \Delta_n \left(\frac{\alpha}{\alpha +
1}\right)
\]
for $n \ge 1$ and
\[
\Delta_1(\alpha) \, = \, -2.
\]
We have used here the fact that $(-1)^{d_n} = 1$, since $d_n = 2
\cdot 3^{n-1}$ is even for all $n \ge 1$.

Consider now a signed composition $\gamma \in \Comp^B(n)$ with a
corresponding word $w \in L_*^n$. It gives rise to three signed
compositions of $n+1$, corresponding to the words $0w, 1w, *w \in
L_*^{n+1}$. Using the notation $m_w(\alpha)$ instead of
$m_\gamma(\alpha)$, it is clear that
\[
m_{0w}(\alpha) \ = \ m_{1w}(\alpha) \ = \ 2^{(k+1)/2} \cdot
(\alpha \cdot 1 + (1 - \alpha)) \cdot \prod_{i = 1}^{k} \gamma_i \
= \ 2^{1/2} \cdot m_w(1),
\]
whereas
\[
m_{*w}(\alpha) \ = \ 2^{k/2} \cdot (\alpha \cdot (\gamma_1 + 1) +
(1 - \alpha)) \cdot \prod_{i = 2}^{k} \gamma_i \ = \ (\alpha + 1)
\cdot m_w \left(\frac{\alpha}{\alpha + 1}\right).
\]
Therefore
\[
m_{0w}(\alpha) \cdot m_{1w}(\alpha) \cdot m_{*w}(\alpha) \ = \
2(\alpha + 1) \cdot m_w(1)^2 \cdot m_w \left(\frac{\alpha}{\alpha
+ 1}\right),
\]
which implies that, at least for $\alpha > 0$, both sides of
equation~(\ref{eq:det_alpha}) satisfy the same recursion (and,
clearly, also the same initial conditions for $n = 1$). The two
sides are therefore equal for all $\alpha > 0$ and (being
polynomials in $\alpha$), are actually equal for all $\alpha$.
\end{proof}

\section{Knuth classes and involutions}
\label{sec:Knuth}

This section discusses irreducible $B_n$-characters again,
describes a different interpretation of Theorem~\ref{thm:roiB} in
terms of Knuth classes of type $B$ and confirms that the set of
involutions in $B_n$ is fine for the character of the Gelfand
model for this group.

We first recall the definition of the natural analogue of the
Robinson--Schensted correspondence for the group $B_n$
\cite[pages~145--146]{Sta82} \cite{SW85}. Let $w = (w(1),
w(2),\dots,w(n)) \in B_n$ and let $(w(a_1), w(a_2),\dots,w(a_k))$
be the subsequence of unbarred elements of $w$. Applying
Schensted's correspondence to the two-line array 
$\left(\begin{array}{cccc} a_1 & a_2 & \cdots & a_k \\ w(a_1) & w(a_2) & \cdots & w(a_k) \end{array}\right)$ 
results in a pair $(P^+, Q^+)$ of Young tableaux of the
same shape $\lambda \vdash k$. Similarly, from the subsequence of
barred elements of $w$ and the corresponding two-line array we get
a pair $(P^-, Q^-)$ of Young tableaux of the same shape $\mu
\vdash n-k$. Then $P^B(w) = (P^+, P^-)$ and $Q^B(w) = (Q^+, Q^-)$
are standard Young bitableaux of shape $(\lambda, \mu) \vdash n$
and the map which assigns the pair $(P^B(w), Q^B(w))$ to $w$ is a
bijection from the group $B_n$ to the set of pairs of standard
Young bitableaux of the same shape and size $n$.

The following two properties of this map are explicit in
\cite[Section~8]{SW85} and implicit in \cite[page~146]{Sta82},
respectively.

\begin{proposition} \label{BRS_properties}
For every $w \in B_n$:
\begin{enumerate}[(a)]
\item $P^B(w^{-1}) = Q^B(w)$, 
\item $\cDes(w) = \cDes(Q^B(w))$.
\end{enumerate}
\end{proposition}

A \emph{Knuth class} of (type $B$ and) shape $(\lambda, \mu)$ is a
set of the form $\{w \in B_n: P^B(w) = T\}$ for some fixed $T \in
\SYT(\lambda, \mu)$ (these Knuth classes should not be confused
with the ones considered in the study of Kazdhan-Lusztig cells of
type $B$; see \cite{Bon10, Pie09}). The first part of the
following corollary, already discussed in Section~\ref{sec:proof},
is a restatement of Theorem~\ref{thm:roiB}. The second part
follows from the first and Proposition~\ref{BRS_properties} (b).

\begin{corollary}
\label{cor:irrB} For every bipartition $(\lambda, \mu) \vdash n$,
the following are fine sets for the irreducible $B_n$-character
$\chi^{\lambda,\mu}$:
  \begin{enumerate}[(i)]
    \item The set of standard Young bitableaux of shape $(\lambda,
               \mu)$.
    \item All Knuth classes of shape $(\lambda, \mu)$.
\end{enumerate}
\end{corollary}

We recall that a \emph{Gelfand model} for a group $G$ is any
representation of $G$ which is equivalent to the multiplicity-free
direct sum of its irreducible representations. The following
proposition is a $B_n$-analogue of the corresponding statement
\cite[Proposition~1.5]{APR08} (see also
\cite[Proposition~3.12~(iii)]{AR15}) for $\mathfrak{S}_n$.

\begin{proposition} \label{prop:GelfandB}
The set of involutions in $B_n$, endowed with the standard signed
descent map, is fine for the character of the Gelfand model of
$B_n$.
\end{proposition}
\begin{proof}
Corollary~\ref{cor:irrB} implies that the set of standard Young
bitableaux of size $n$ is fine for the character of the Gelfand
model of $B_n$. Moreover, as a direct consequence of
Proposition~\ref{BRS_properties}, the signed descent set is
equidistributed over the set of standard Young bitableaux of size
$n$ and the set of involutions in $B_n$. The proof follows from
these two statements.
\end{proof}

\begin{example}
We confirm this statement for $n=2$. The involutions in 
$B_2$ are the signed permutations $(1, 2)$, $(2, 1)$, $(\bar{1}, 2)$, $(1, \bar{2})$, $(\bar{1}, \bar{2})$ and 
$(\bar{2}, \bar{1})$. Therefore,
\begin{multline*}
  \sum_{w \in B_2: \, w^{-1} = w} F_{\cDes(w)} (x, y) \\ = 
  F_{(2)}(x,y) \, + \, F_{(1,1)}(x,y) \, + \, F_{(\bar{1},1)}(x,y)
  \, + \, F_{(1,\bar{1})}(x,y) \ +
  F_{(\bar{2})}(x,y) \, + \, F_{(\bar{1},\bar{1})} (x,y) \\
  = s_{(2)}(x) \, + \, s_{(1,1)}(x) \, + \, s_{(1)}(x) s_{(1)}
      (y) \, + \, s_{(2)}(y) \, + \, s_{(1,1)}(y) \\
  = \ch \ ( \chi^{((2),\varnothing)} \, + 
      \, \chi^{((1,1),\varnothing)} \, + \, \chi^{((1),(1))}
      \, + \, \chi^{(\varnothing,(2))} \, + \, 
      \chi^{(\varnothing,(1,1))}), 
\end{multline*}
where we have indexed functions $F_\sigma(x,y)$ by signed 
compositions, rather than signed sets, and the second equality 
follows by direct computation or use of 
Proposition~\ref{prop:schurprod}.
\end{example}

An \emph{inverse signed descent class} in $B_n$ is a set of the
form $\{ w \in B_n: \cDes(w^{-1}) = \sigma\}$ for some $\sigma \in
\Sigma^B(n)$. We postpone the definition of flag-major index to
Section~\ref{sec:flag-inv}, where the representation of $B_n$
corresponding to the fine set (ii) in the following proposition
will also be described.

\begin{proposition}
\label{prop:ISDC} The following subsets of $B_n$ are fine for some
$B_n$-characters:
\begin{enumerate}[(i)]
\item All inverse signed descent classes. \item The set
of elements of $B_n$ whose inverses have a given flag-major index.
\end{enumerate}
\end{proposition}
\begin{proof}
By Proposition~\ref{BRS_properties}, the signed descent set
$\cDes$ is fixed over an inverse Knuth class. It follows that
inverse signed descent classes are unions of Knuth classes, which
are fine sets by Corollary~\ref{cor:irrB}, and hence are fine sets
as well. Morever, subsets of $B_n$ with fixed signed descent set
have fixed flag-major index. As a result, the set of elements of
$B_n$ whose inverses have a given flag-major index is a union of
inverse signed descent classes and hence is fine, as a disjoint
union of fine sets.
\end{proof}

\section{Coinvariant algebra and flag statistics}
\label{sec:flag-inv}

This section provides a $B_n$-analogue of a result essentially due
to Roichman~\cite{Roi00}, which describes explicitly fine sets for
the characters of the action of $\mathfrak{S}_n$ on the
homogeneous components of the coinvariant algebra of type $A$.

Throughout this section we denote by $P_n$ the polynomial ring
$F[x_1, \ldots, x_n]$ in $n$ variables over a field $F$ of
characteristic zero. The symmetric group $\mathfrak{S}_n$ acts on
$P_n$ by permuting the  variables. Let $I_n$ be the ideal of $P_n$
generated by the $\mathfrak{S}_n$-invariant (symmetric)
polynomials with zero constant term. The group $\mathfrak{S}_n$
acts on the quotient ring $P_n/I_n$, known as the
\emph{coinvariant algebra} of $\mathfrak{S}_n$, and the resulting
representation is isomorphic to the regular representation; see,
for instance, \cite[Section~3.6]{Hum90}. Since $P_n$ and $I_n$ are
naturally graded by degree and the action of $\mathfrak{S}_n$
respects this grading, the coinvariant algebra is also graded by
degree and $\mathfrak{S}_n$ acts on each homogeneous component.
Denote by $\chi_{n, k}$ the character of the
$\mathfrak{S}_n$-action on the $k$-th homogeneous component of
$P_n/I_n$, for $0 \le k \le \binom{n}{2}$.

Let $w \in \mathfrak{S}_n$ be a permutation. Recall \cite[Sections~1.3--1.4]{StaEC1} that $\inv(w)$ denotes the number
of inversions of $w$, $\Des(w)$ its set of descents, and $\maj(w)$
its major index (the sum of the elements of $\Des(w)$).
Note also that $\inv(w) = \inv(w^{-1})$ for $w \in \mathfrak{S}_n$.
\begin{theorem} {\rm (\cite{Roi00} \cite[Theorem~3.7]{AR15})}
\label{thm:coinvA} For $0 \le k \le \binom{n}{2}$, each of the
following subsets of $\mathfrak{S}_n$, endowed with the standard
descent map, is a fine set for the $\mathfrak{S}_n$-character
$\chi_{n, k}$:
\begin{enumerate}[(i)]
\item 
$\{w \in \mathfrak{S}_n: \inv(w^{-1}) = k\}$; 
\item
$\{w \in \mathfrak{S}_n: \maj(w^{-1}) = k\}$.
\end{enumerate}
\end{theorem}
\begin{proof}
The claim for (i) is a restatement of \cite[Theorem~1]{Roi00}. 
The claim for (ii) follows from the one for (i) by the fact~\cite[Theorem~1]{FS78} that the number of $w \in \mathfrak{S}_n$ with $\Des(w^{-1}) = S$ and $\inv(w) = k$ is 
equal to that of $w \in \mathfrak{S}_n$ with $\Des(w^{-1}) = S$ 
and $\maj(w) = k$ for all $S \subseteq [n-1]$ and $k$.
\end{proof}

Consider now the group $B_n$. It acts on the polynomial ring $P_n$
by permuting the variables $x_1, \ldots,x_n$ and flipping their
signs. Let $I^B_n$ be the ideal of $P_n$ generated by the
$B_n$-invariant polynomials (i.e., symmetric functions in the
squares $x_1^2, \ldots, x_n^2$) with zero constant term. The
\emph{coinvariant algebra} of $B_n$ is the quotient ring
$P_n/I^B_n$. 
$B_n$ acts on each of its
homogeneous components, and the resulting
representation is isomorphic to the regular representation of $B_n$. 
It is graded by degree; denote by $\chi^B_{n, k}$ the character of the $B_n$-action on the $k$-th homogeneous component of $P_n/I^B_n$, for $0 \le k \le n^2$.

The \emph{flag-major index} \cite{AR01} of a signed permutation 
$w \in B_n$ is defined as
\[
\fmaj(w) \ = \ 2 \cdot \maj(w) + \barr(w),
\]
where $\maj(w)$ is the sum of the elements of $\Des(w)$ (as
defined in Section~\ref{subsec:perm}) and $\barr(w)$ is the number
of indices $i \in [n]$ for which $w(i)$ is barred. The
\emph{flag-major index} of a standard Young bitableau $Q$ of shape
$(\lambda, \mu)$ is defined as twice the sum of the elements of $\Des(Q)$ 
plus the size of $\mu$, so that $\fmaj(w) = \fmaj
(Q^B(w))$ for every $w \in B_n$, in the notation of Section~\ref{sec:Knuth}, by Proposition~\ref{BRS_properties} (b). The \emph{flag-inversion number}
of $w \in B_n$ is defined as
\[
\finv(w) \ = \ 2 \cdot \inv(w) + \barr(w),
\]
where $\inv(w)$ is the number of inversions of the sequence
$(w(1), w(2),\ldots,w(n))$ with respect to the total order
(\ref{eq:total}) (this is a variant of the notion of 
flag-inversion number introduced in~\cite{FH07}; see 
also~\cite{ABR11, Fire04} and references therein).

The equidistribution result~\cite[Theorem~1]{FS78}, mentioned 
in the proof of Theorem~\ref{thm:coinvA}, follows from a theorem
of Foata~\cite{Foa68}, which implies that $\inv$ and $\maj$ are
equidistributed over the set of permutations of an 
ordered multiset 
(see, for instance, the discussion 
in~\cite[Section~1]{FH07}), by an application of the inclusion-exclusion principle. A similar 
argument yields the following analogous result for $B_n$.
\begin{proposition}\label{lem:finv_fmaj}
For every $n$, $k$ and $\sigma \in \Sigma^B(n)$,
the number of $w\in B_n$ with $\cDes(w^{-1}) = \sigma$ and $\finv(w)=k$ is equal to 
the number of $w \in B_n$ with $\cDes(w^{-1}) = \sigma$ and $\fmaj(w) = k$.
\end{proposition}

\begin{proof}
Let us denote by $f_{n,k}(\sigma)$ (respectively, $g_{n,k}(\sigma)$)
the number of $w \in B_n$ with $\cDes(w^{-1}) = \sigma$ and $\inv
(w) = k$ (respectively, $\maj(w) = k$). Since the function 
$\barr(w)$ is constant within each class $\{ w \in B_n: \, 
\cDes(w^{-1}) = \sigma \}$, we have to show that $f_{n,k}(\sigma) 
= g_{n,k} (\sigma)$ for all $n$, $k$ and $\sigma \in \Sigma^B(n)$. 

Fix $n$, $k$ and $\sigma \in \Sigma^B(n)$.
Let $\gamma \in \Comp^B(n)$, with (absolute) parts $\gamma_1, \ldots, \gamma_t$, be the signed composition corresponding to $\sigma$.
Consider the multiset $\mM_\sigma$ obtained by breaking the set $[n]$ into contiguous segments of lengths $\gamma_1, \ldots, \gamma_t$,  
replacing all the entries of a segment by copies of its smallest entry,
and barring them if the corresponding part of $\gamma$ is barred. 
For example, if $\gamma = (2,2,\bar{2}, \bar{2})$ then $\mM_\sigma = \{1,1,3,3,\bar{5},\bar{5},\bar{7},\bar{7}\}$. 
We consider the ground set of the multiset $\mM_\sigma$ totally ordered by (\ref{eq:total}). For $\tau \in \Sigma^B(n)$ we write $\tau \preceq \sigma$ if 
the signed composition corresponding to $\tau$ can be obtained from the one corresponding to $\sigma$ by replacing consecutive unbarred parts, or consecutive barred parts, by their sum, while keeping the bar when present. 
For example, for $\gamma = (2,2,\bar{2},\bar{2})$,
the signed sets $\tau \preceq \sigma$ are those with signed 
compositions $\gamma$, $(4,\bar{2},\bar{2})$, $(2,2,\bar{4})$ 
and $(4,\bar{4})$. A little thought shows that the aforementioned equidistribution result of Foata~\cite{Foa68}, applied to the 
ordered multiset $\mM_\sigma$, implies that
\[ \sum_{\tau \preceq \sigma} f_{n,k} (\tau) \ = \ 
   \sum_{\tau \preceq \sigma} g_{n,k} (\tau) \]
for every $\sigma \in \Sigma^B(n)$. Using the principle of 
inclusion-exclusion, we conclude that $f_{n,k}(\sigma) = g_{n,k}
(\sigma)$ for every $\sigma \in \Sigma^B(n)$ and the proof 
follows.
\end{proof}

The following statement is our $B_n$-analogue of
Theorem~\ref{thm:coinvA}.

\begin{theorem} \label{thm:coinvB}
For $0 \le k \le n^2$, each of the following subsets of $B_n$,
endowed with the standard signed descent map, is a fine set for
the $B_n$-character $\chi^B_{n, k}$:
\begin{enumerate}[(i)]
\item $\{w \in B_n: \finv(w^{-1}) = k\}$; \item $\{w
\in B_n: \fmaj(w^{-1}) = k\}$.
\end{enumerate}
\end{theorem}
\begin{proof}
For every bipartition $(\lambda, \mu) \vdash n$, the multiplicity
of the irreducible character $\chi^{\lambda, \mu}$ in $\chi^B_{n,
k}$ was shown by Stembridge (see~\cite[Theorem~5.3]{Ste89}) to be
equal to the number of standard Young bitableaux of shape
$(\lambda, \mu)$ and flag-major index $k$. This fact,
Corollary~\ref{cor:irrB} (i) and basic properties of the
Robinson--Schensted correspondence for $B_n$ (see
Proposition~\ref{BRS_properties}) imply the claim for (ii). The
claim for (i) follows from that for (ii) by Proposition~\ref{lem:finv_fmaj}.
\end{proof}

\begin{example}
For $n=k=3$ we have
\begin{multline*}
  \{ w \in B_3: \, \finv(w) = 3\} \\ = \{ (\bar{1},\bar{2},\bar{3}), 
  (1,\bar{2},3), (1,\bar{3},2), (2,\bar{1},3), (\bar{1},3,2), 
  (\bar{2},3,1), (\bar{3},2,1) \}, \end{multline*}
\begin{multline*}
\{w \in B_3: \, \fmaj(w) = 3\} \\ = \{ (\bar{1},\bar{2},\bar{3}), 
  (1,\bar{2},3), (1,\bar{3},2), (2,\bar{1},3), (3,\bar{1},2),
  (3,\bar{2},1), (2,\bar{3},1) \}, \end{multline*}
\begin{multline*}  \{w \in B_3: \, \finv(w^{-1}) = 3\} \\ = \{ (\bar{1},\bar{2},\bar{3}), 
  (1,\bar{2},3), (1,3,\bar{2}), (\bar{2},1,3), (\bar{1},3,2), 
  (3,\bar{1},2), (3,2,\bar{1}) \}, \end{multline*}
\begin{multline*}  
  \{w \in B_3: \fmaj(w^{-1}) = 3\} \\ = \{ (\bar{1},\bar{2},\bar{3}), 
  (1,\bar{2},3), (1,3,\bar{2}), (\bar{2},1,3), (\bar{2},3,1),
  (3,\bar{2},1),  (3,1,\bar{2}) \}. 
\end{multline*}
 
The signed descent compositions of the elements of either of the last 
 two sets are $(\bar 3)$, $(1, \bar 1, 1)$, $(2, \bar 1)$,
 $(\bar 1, 2)$, $(\bar 1, 1,1)$, $(1, \bar 1, 1)$ and $(1,1, \bar 1)$
and hence
\begin{eqnarray*}
  \sum_{\substack{w \in B_3 \\ \finv(w^{-1}) = 3}} F_{\cDes(w)} (x, y) &=&  
  \sum_{\substack{w \in B_3 \\ \fmaj(w^{-1}) = 3}} F_{\cDes(w)} (x, y) \\
  &=&  F_{(\bar 3)}(x,y)\ + \\
  & &  F_{(1, \bar 1, 1)}(x,y) \, + \, F_{(2, \bar 1)}(x,y)   \, + \, 
  F_{(\bar 1, 2)}(x,y) \ + \\
  & & F_{(\bar 1, 1,1)}(x,y) \, + \, F_{(1,\bar{1},1)} (x,y)\, + \, 
  F_{(1,1,\bar{1})} (x,y) \\
  &=& s_{(3)}(y) \, + \, s_{(2)}(x)s_{(1)}(y) \, + \, s_{(1,1)}(x)s_{(1)}(y).
 \end{eqnarray*}
The last equality follows by direct computation or use of 
Proposition~\ref{prop:schurprod}.
Indeed, there are exactly three bitableaux of size $3$ and
flag-major index $3$, namely
\[
( \, \varnothing \ , \ \young(123) \ ), \ \ \ \ ( \ \young(1,3) \ ,
\ \young(2) \ ), \quad ( \ \young(13) \ , \ \young(2) \ ).
\]
Thus, by~\cite[Theorem~5.3]{Ste89},
\begin{eqnarray*}
\ch\left(\chi^B_{3,3}\right) &=& \ch\left(\chi^{(\varnothing, (3))} + \chi^{((2), (1))} + \chi^{((1,1),
(1))}\right) \\ &=& s_{(3)}(y)+s_{(1)}(y)s_{(2)}(x)+s_{(1)}(y)s_{(1,1)}(x).\end{eqnarray*}
\end{example}

\begin{remark} \label{rem:coinv} \rm
Another natural $B_n$-analogue of the inversion number statistic
$\inv: \mathfrak{S}_n \to \NN$ is the length function $\ell_B: B_n
\to \NN$, in terms of the Coxeter generators of $B_n$; see
\cite[Section~8.1]{BB05} for an explicit combinatorial description
of this function. A formula of Adin, Postnikov and
Roichman~\cite[Theorem~4]{APR01} implies that, for every $k$, the
subset $\{w \in B_n: \ell_B(w) = k\}$ of $B_n$ is a
$\mathfrak{S}_n$-fine set for the restriction of $\chi^B_{n,k}$ to
$\mathfrak{S}_n$. However, this set is not fine for any
$B_n$-character. Indeed, let us write $\Neg(\sigma)$ for the set
of negative coordinates of the sign vector of $\sigma \in
\Sigma^B(n)$. Given a fine set $\bB$ for some $B_n$-character,
condition (iii) of Corollary~\ref{cor:main} implies that for any
$J\subseteq [n]$, the number of elements $b \in \bB$ with
$\Neg(\cDes(b)) = J$ depends only on the cardinality $J$ and is
divisible by $\binom{n}{|J|}$. We leave it to the interested
reader to verify that this restriction is violated in the case
considered here.
\end{remark}

\section{Conjugacy classes}
\label{sec:conj}

This section reviews a result of Poirier~\cite{Poi98}, which
implies that every conjugacy class in $B_n$ is a fine set, and
discusses in detail two interesting examples of fine sets which
are unions of conjugacy classes, namely the set of derangements in
$B_n$ and the set of $k$-roots of the identity element.

Consider the alphabet $\aA = \{x_1, x_2,\dots,y_1, y_2,\dots\}$. A
\emph{necklace} of length $m$ over $\aA$ is an equivalence class
of words of length $m$ over $\aA$, where two such words are
equivalent if one is a cyclic shift of the other. A necklace is
called \emph{primitive} if the corresponding word (the choice of
representative being irrelevant) is not equal to a power of a word
of smaller length. Given a finite multiset $\mM$ of primitive
necklaces over $\aA$, the product of all variables appearing in
the necklaces of $\mM$ is called the \emph{evaluation} of $\mM$.
The \emph{signed cycle type} of $\mM$ is the bipartition $(\alpha,
\beta)$ for which the parts of $\alpha$ (respectively, $\beta$)
are the lengths of the necklaces of $\mM$ which contain an even
(respectively, odd) number of $y$ variables. We will denote by
$L^B_{\alpha,\beta} (x, y)$ the formal sum of the evaluations of
all multisets of primitive necklaces of signed cycle type
$(\alpha, \beta)$ over $\aA$.

The following result extends to $B_n$ analogous results for
$\mathfrak{S}_n$ (see, for instance, the discussion in
\cite[Section~2]{GR93}). Combined with Theorem~\ref{thm:main}, it
shows that each conjugacy class $C_{\alpha, \beta}$, endowed with
the standard colored descent map, is a fine set for $B_n$.

\begin{theorem} {\rm (\cite[Theorem~16]{Poi98})}
\label{thm:poirier} For every bipartition $(\alpha,\beta) \vdash
n$
 \begin{equation} \label{eq:poirier}
    \sum_{w \in C_{\alpha, \beta}} F_{\cDes(w)} (x, y) \ = \
    L^B_{\alpha, \beta} (x, y).
  \end{equation}
Moreover, $L^B_{\alpha, \beta} (x, y)$ is equal to the Frobenius
characteristic of a representation of $B_n$.
\end{theorem}

The representation of $B_n$ which appears in the theorem is a
$B_n$-analogue of the Lie representation of given type of
$\mathfrak{S}_n$. We omit the definition, since this
representation does not play any major role in this paper, and
refer to \cite[Section~4]{Poi98} and references therein for more
information.

\subsection{Derangements}
\label{subsec:derange}

The set of derangements in $B_n$, being a union of conjugacy
classes, is naturally a fine set. This section describes the
corresponding representation of $B_n$ and its decomposition as a
direct sum of irreducible representations, and expresses its
Frobenius characteristic in terms of another set of elements of
$B_n$, which may be called desarrangements of type $B$. These
results are essentially $B_n$-analogues of results of
D\'esarm\'enien and Wachs~\cite{DeWa88, DeWa93} and of Reiner and
Webb~\cite[Section~2]{ReiW04} for the symmetric group, which will
first be reviewed; they provided much of the motivation behind the
present paper. A $B_n$-analogue of the equidistribution
\cite{DeWa88, DeWa93} of the descent set among derangements and
desarrangements in $\mathfrak{S}_n$ will be deduced.

The set of derangements (elements without fixed points) in
$\mathfrak{S}_n$, denoted here by $\dD_n$, is a fine set for
$\mathfrak{S}_n$ by Theorem~\ref{thm:mainA} and
\cite[Theorem~3.6]{GR93}. The corresponding
$\mathfrak{S}_n$-representation, first studied in
\cite[Section~2]{ReiW04}, can be described as follows. We denote
by $\iI_n$ the set of words from the alphabet $[n]$ with no
repeated letters (known as injective words), partially ordered by
setting $u \le v$ for $u, v \in \iI_n$ if $u$ is a subword of $v$.
Thus, the empty word is the minimum element and the $n!$
permutations in $\mathfrak{S}_n$ are the maximal elements of
$\iI_n$. The poset $\iI_n$ is the face poset of a regular cell
complex $\kK_n$ whose faces are combinatorially isomorphic to
simplices. The symmetric group $\mathfrak{S}_n$ acts naturally on
$\iI_n$ and hence on the augmented cellular chain complex of
$\kK_n$ over $\CC$; see \cite{ReiW04} and references therein for
more explanation and for background on regular cell complexes.
Following \cite{ReiW04}, we denote by $\chi_n$ the character of
the resulting representation on the top reduced homology
$\tilde{H}_{n-1} (\kK_n, \CC)$.

Using the Hopf trace formula and the shellability of $\kK_n$, it
was shown in \cite[Proposition~2.1]{ReiW04} that
  \begin{equation} \label{eq:RWrec}
    \chi_n \ = \ \sum_{k=0}^n \, (-1)^{n-k} \ 1 \!
                \uparrow^{\mathfrak{S}_n}_{(\mathfrak{S}_1)^k \times
    \mathfrak{S}_{n-k}}
  \end{equation}
for every $n$ (the reader who is unfamiliar with cellular homology
may wish to consider this formula as the definition of $\chi_n$).

We denote by $\eE_n$ the set of permutations (called
desarrangements in \cite{DeWa93}) $w \in \mathfrak{S}_n$ for which
the smallest element of $[n] \sm \Des(w)$ is even. The following
theorem, which combines results of \cite{DeWa88, DeWa93} and
\cite[Section~2]{ReiW04}, shows that $\dD_n$ and $\{w^{-1}: w \in
\eE_n\}$ are fine sets for the sign twist $\varepsilon_n \otimes
\chi_n$ of $\chi_n$ and determines its decomposition as a direct
sum of irreducible characters.

\begin{theorem} {\rm (\cite{DeWa88, DeWa93} \cite[Section~2]{ReiW04})} \label{thm:DRWW}
For every positive integer $n$,

\begin{eqnarray*}
  \ch(\varepsilon_n \otimes \chi_n) &=& \sum_{w \in \dD_n}
  F_{n, \Des(w)} (x) \ = \ \sum_{w \in \eE_n} F_{n, \Des(w^{-1})} (x)
  \\ &=& \sum_{\lambda \vdash n} c_\lambda s_\lambda(x)
\end{eqnarray*}
where $c_\lambda$ is the number of standard Young
tableaux $P$ of shape $\lambda$ for which the smallest element of
$[n] \sm \Des(P)$ is even. In particular, the number of
derangements $w \in \dD_n$ with $\Des(w) = S$ is equal to the
number of permutations $w \in \eE_n$ with $\Des(w^{-1}) = S$ for
every $S \subseteq [n-1]$.
\end{theorem}

We now turn attention to the hyperoctahedral group $B_n$. We
denote by $\iI_n^B$ the set of words from the alphabet $\Omega_n$
which contain no two letters with equal absolute values and
partially order this set by the subword order. The poset $\iI_n^B$
is the face poset of a regular cell complex $\kK_n^B$ whose faces
are combinatorially isomorphic to simplices; see
Figure~\ref{fig:KB2} for the case $n=2$. The group $B_n$ acts
naturally on $\iI_n^B$ and hence on the augmented cellular chain
complex of $\kK_n^B$ over $\CC$. We denote by $\psi_n$ the
character of the resulting representation on the top reduced
homology $\tilde{H}_{n-1} (\kK_n^B, \CC)$.

  \begin{figure}[htb]
  \includegraphics[width=\textwidth]{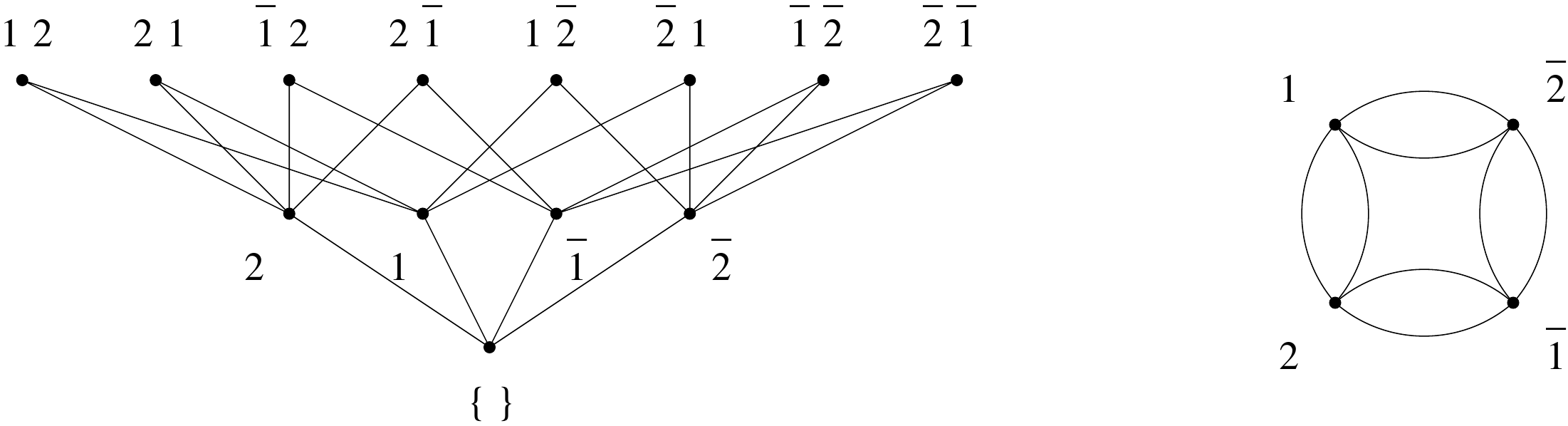}
  \caption{The face poset of the cell complex $\kK_2^B$.}
  \label{fig:KB2}
  \end{figure}

We denote by $\dD_n^B$ the set of derangements (elements without
fixed points, when thought of as permutations of $\Omega_n$) in
$B_n$ and by $\eE^B_n$ the set of $w \in B_n$ for which the
maximum number $k$ such that $w(1) > w(2) > \cdots > w(k) > 0$ is
even, possibly equal to zero (where for $a \in \Omega_n$ we write
$a > 0$ if $a$ is unbarred). Finally, we denote by $\omega_x$ the
standard involution on $\Lambda(x) \otimes \Lambda(y)$ acting on
the $x$ variables. The following statement is our $B_n$-analogue
of Theorem~\ref{thm:DRWW}.

\begin{theorem} \label{thm:derB}
For every positive integer $n$,

\begin{eqnarray*}
  \omega_x \, \ch(\psi_n) &=& \sum_{w \in \dD^B_n} F_{\cDes(w)} (x, y)
  \ = \ \sum_{w \in \eE^B_n} F_{\cDes(w^{-1})} (x, y) \\
  &=& \sum_{(\lambda, \mu) \vdash n} c_{\lambda, \mu} s_\lambda(x)
      s_\mu (y)
\end{eqnarray*}
where $c_{\lambda, \mu}$ is the number of standard Young
bitableaux $(P^+, P^-)$ of shape $(\lambda, \mu)$ such that the
largest number $k$ for which $1, 2,\dots,k$ appear in the first
column of $P^+$ is even (possibly equal to zero).

In particular, the number of derangements $w \in \dD^B_n$ with
$\cDes(w) = \sigma$ is equal to the number of signed permutations
$w \in \eE^B_n$ with $\cDes(w^{-1}) = \sigma$ for every $\sigma
\in \Sigma^B(n)$.
\end{theorem}

\begin{example}
To illustrate the theorem, let us compute explicitly these 
expressions for $\omega_x \, \ch(\psi_n)$ when $n=2$. The 
bitableaux which satisfy the condition in the statement of
the theorem are
\[ ( \ \young(1,2) \ , \ \varnothing \, ), \ \ \
   ( \ \young(2) \ , \ \young(1) \ ), \ \ \
   ( \, \varnothing \ , \ \young(12) \ ), \ \ \
   ( \, \varnothing \ , \ \young(1,2) \ ). \]
As a result, the third expression claimed for $\omega_x \, \ch(\psi_n)$ 
by the theorem gives
\[ \omega_x \, \ch(\psi_2) \ = \ s_{(1,1)}(x) \, + \,
s_{(1)}(x) s_{(1)}(y) \, + \, s_{(2)}(y) \, + \, s_{(1,1)}(y).
\]
Moreover, the signed descent sets of the elements of
$$\dD^B_2 = \{ (2,1), (2, \bar{1}), (\bar{2}, 1), (\bar{2}, \bar{1}), 
(\bar{1}, \bar{2}) \},$$ written as signed compositions, are $(1,1)$, 
$(1,\bar{1})$, $(\bar{1},1)$, $(\bar{1},\bar{1})$ and $(\bar{2})$ 
and the same holds for those of $\{ w^{-1}: w \in \eE^B_2 \} = \{ (2,1),
(\bar{1}, 2), (\bar{1}, \bar{2}), (2, \bar{1}), (\bar{2}, \bar{1}) \}$. 
Hence, either of the first two expressions claimed for 
$\omega_x \, \ch(\psi_n)$ gives
\[ \omega_x \, \ch(\psi_2) \ = \ F_{(1,1)}(x,y) \, + \,
F_{(1,\bar{1})}(x,y) \, + \, F_{(\bar{1},1)}(x,y) \, + \,
F_{(\bar{2})}(x,y) \, + \, F_{(\bar{1},\bar{1})}(x,y). \]
By direct computation or use of
Proposition~\ref{prop:schurprod}, one can verify that the
two formulas for $\omega_x \, \ch(\psi_2)$ are equivalent.
\end{example}

The proof of part of Theorem~\ref{thm:derB} will be based on the 
following proposition. The proof of the proposition is a direct 
analogue of the proofs of Propositions~2.1 and 2.2 in
\cite[Section~2]{ReiW04}.

\begin{proposition} \label{prop:psirec}
For every positive integer $n$,
  \begin{equation} \label{eq:psirec1}
    \psi_n \ = \ \sum_{k=0}^n \, (-1)^{n-k} \ 1 \!
                \uparrow^{B_n}_{(\mathfrak{S}_1)^k \times B_{n-k}}.
  \end{equation}
For every $n \ge 2$,
  \begin{equation} \label{eq:psirec2}
    \psi_n \ = \ (1 \otimes \psi_{n-1}) \!
    \uparrow^{B_n}_{\mathfrak{S}_1 \times B_{n-1}} + \
    (-1)^n \, 1_{B_n}.
  \end{equation}
\end{proposition}
\begin{proof}
Following arguments in \cite[Section~2]{ReiW04}, we observe that
the words of given rank $k$ in $\iI_n^B$ index the elements of a
basis of the augmented cellular chain group $C_{k-1} (\kK_n^B,
\CC)$. The group $B_n$ acts transitively on these words with
stabilizer isomorphic to the Young subgroup $(\mathfrak{S}_1)^k
\times B_{n-k}$. As a result, the character of $B_n$ acting on
$C_{k-1} (\kK_n^B, \CC)$ is equal to the induced character $1 \!
\uparrow^{B_n}_{(\mathfrak{S}_1)^k \times B_{n-k}}$.
Equation~(\ref{eq:psirec1}) then follows from the Hopf trace
formula and the shellability (see \cite[Theorem~1.2]{JW09}) of
$\kK_n^B$. Finally, using (\ref{eq:psirec1}) we compute that
\begin{eqnarray*}
  \psi_n &=& \sum_{k=0}^n \, (-1)^{n-k} \ 1 \!
                \uparrow^{B_n}_{(\mathfrak{S}_1)^k \times B_{n-k}} \\
  &=& (-1)^n \, 1_{B_n} \, + \, \sum_{k=1}^n \, (-1)^{n-k} \ 1 \!
                \uparrow^{B_n}_{(\mathfrak{S}_1)^k \times B_{n-k}} \\
  &=& (-1)^n \, 1_{B_n} \, + \, \left( 1 \otimes \sum_{k=1}^n \,
      (-1)^{n-k} \ 1 \! \uparrow^{B_{n-1}}_{(\mathfrak{S}_1)^{k-1}
      \times B_{n-k}} \right) \uparrow^{B_n}_{\mathfrak{S}_1 \times
      B_{n-1}} \\
  & & \\
  &=& (-1)^n \, 1_{B_n} \, + \, (1 \otimes \psi_{n-1}) \!
      \uparrow^{B_n}_{\mathfrak{S}_1 \times B_{n-1}}.
\end{eqnarray*}
This verifies Equation~(\ref{eq:psirec2}).
\end{proof}

\begin{proof}[Proof of Theorem~\ref{thm:derB}] Applying the
Frobenius characteristic to Equation~(\ref{eq:psirec2}) and using
Lemma~\ref{lem:ch1B} gives
  $$ \ch(\psi_n) \ = \ s_1(x, y) \, \ch(\psi_{n-1}) \, + \, (-1)^n
     s_n(x) $$
and hence
  \begin{equation} \label{eq:psirec3}
    \omega_x \, \ch(\psi_n) \ = \ s_1(x, y) \cdot \omega_x \,
    \ch(\psi_{n-1}) \, + \, (-1)^n  e_n(x)
  \end{equation}
for $n \ge 2$, where $s_1 (x, y) = s_1 (x) + s_1 (y)$. One can
check that $\psi_1$ is the sign character of $B_1$, so that
$\ch(\psi_1) = \omega_x \, \ch(\psi_1) = s_1 (y)$. By this
observation and an easy induction argument, the recurrence
(\ref{eq:psirec3}) and Pieri's rule imply that
  \begin{equation} \label{eq:psi1}
    \omega_x \, \ch(\psi_n) \ = \ \sum_{(\lambda, \mu) \vdash n}
    c_{\lambda, \mu} s_\lambda(x) s_\mu (y),
  \end{equation}
where $c_{\lambda, \mu}$ is as in the statement of
Theorem~\ref{thm:derB}. Furthermore, using
Proposition~\ref{prop:schurprod} and basic properties of the
Robinson--Schensted correspondence of type $B$ (see
Proposition~\ref{BRS_properties}) we get
\begin{eqnarray}
  \sum_{(\lambda, \mu) \vdash n} c_{\lambda, \mu} s_\lambda(x)
  s_\mu (y) &=& \sum_{(\lambda, \mu) \vdash n} c_{\lambda, \mu}
  \sum_{Q \in \SYT(\lambda, \mu)} F_{\cDes(Q)} (x, y) \nonumber \\
  & & \nonumber \\
  &=& \sum_{P \in \SYT'(\lambda, \mu)} \ \sum_{Q \in \SYT(\lambda,
  \mu)} F_{\cDes(Q)} (x, y) \nonumber \\
  & & \nonumber \\
  &=& \sum_{w^{-1} \in \eE^B_n} F_{\cDes(w)} (x, y), \label{eq:psi2}
\end{eqnarray}
where $\SYT'(\lambda, \mu)$ stands for the set of
standard Young bitableaux $P = (P^+, P^-)$ such that the largest
number $k$ for which $1, 2,\dots,k$ appear in the first column of
$P^+$ is even (possibly equal to zero).

Finally, we set
  $$ D^B_n (x, y) \ := \ \sum_{w \in \dD^B_n} F_{\cDes(w)} (x, y) $$
for $n \ge 1$. Since $\dD^B_n$ is a union of conjugacy classes,
Theorem~\ref{thm:poirier} implies that
  $$ D^B_n (x, y) \ = \ \sum L^B_{\alpha, \beta} (x, y) $$
where the sum ranges over all bipartitions $(\alpha, \beta) \vdash
n$ such that $\alpha$ has no part equal to one. We now recall (see
\cite[Section~3]{GR93} and references therein) that there is an
evaluation and length preserving bijection from the set of words
on the alphabet $\aA$ to the set of multisets of primitive
necklaces on $\aA$. Thus, just as in the symmetric group case (see
the proof of \cite[Theorem~8.1]{GR93}), summing $L^B_{\alpha,
\beta} (x, y)$ over all $(\alpha, \beta) \vdash n$ and considering
the number of parts of $\alpha$ equal to one we get
  $$ s_1(x, y)^n \ = \ \sum_{k=0}^n s_k(x) \, D^B_{n-k} (x, y) $$
for $n \ge 0$, where $s_0 (x) = D^B_0 (x, y) = 1$. Equivalently,
we have
\begin{equation} \label{eq:DnB}
  D^B_n (x, y) \ = \ \sum_{k=0}^n \, (-1)^k e_k(x) \, s_1(x, y)^{n-k}
\end{equation}
for $n \ge 1$. This formula and (\ref{eq:psirec3}) show that
  \begin{equation} \label{eq:psi3}
    D^B_n (x, y) \ = \ \omega_x \, \ch(\psi_n)
  \end{equation}
for every $n \ge 1$. The proof follows by combining
Equations~(\ref{eq:psi1}), (\ref{eq:psi2}) and (\ref{eq:psi3}).
\end{proof}

\begin{remark} \label{rem:derB}
The last statement of Theorem~\ref{thm:DRWW} was originally proven
in \cite{DeWa88}. A bijective proof was later provided in
\cite{DeWa93}. An analogous proof of the last statement of
Theorem~\ref{thm:derB} should be possible.
\end{remark}

\subsection{$k$-roots}
\label{subsec:roots}

Given a positive integer $k$, we will denote by $r_{n, k} (w)$ the
number of $k$-roots of $w \in \mathfrak{S}_n$ (meaning, elements
$u \in \mathfrak{S}_n$ with $u^k = w$) and by $r^B_{n, k} (w)$ the
number of $k$-roots of $w \in B_n$. Clearly, $r_{n, k}$ and
$r^B_{n, k}$ are class functions on $\mathfrak{S}_n$ and $B_n$,
respectively. It was shown by Scharf~\cite{Sch91a, Sch91b} that
these functions are actually (non-virtual) characters of
$\mathfrak{S}_n$ and $B_n$. The first statement in the next
theorem follows from \cite[Theorem~1.1]{Roi14} and its proof; the
second statement follows from the first and
Theorem~\ref{thm:mainA}.

\begin{theorem} {\rm (\cite{Roi14})} \label{thm:krootsA}
The set $\{w \in \mathfrak{S}_n: w^k = e\}$ of all $k$-roots of
the identity element in $\mathfrak{S}_n$ is a fine set for the
character $r_{n, k}$. Equivalently,
 \begin{equation} \label{eq:krootsA}
    \ch(r_{n, k}) \ = \ \sum_{w \in \mathfrak{S}_n: \, w^k = e}
    F_{n, \Des(w)} (x)
  \end{equation}
for all positive integers $n, k$.
\end{theorem}

The following theorem gives a partial answer to
\cite[Question~3.4]{Roi14}. It is expected that this theorem can
be extended to the (possibly more complicated) case of even
positive integers $k$, which is left open in the present writing
(the case $k=2$ follows from Proposition~\ref{prop:GelfandB},
since $r^B_{n,2}$ is equal to the character of the Gelfand model
of $B_n$; see~\cite[p.~58]{Isa76}).

\begin{theorem} \label{thm:krootsB}
The set $\{w \in B_n: w^k = e\}$ of all $k$-roots of the identity
element in $B_n$ is a fine set for the character $r^B_{n, k}$ for
every odd positive integer $k$ and all $n \ge 1$. Equivalently,
 \begin{equation} \label{eq:krootsB}
    \ch(r^B_{n, k}) \ = \ \sum_{w \in B_n: \, w^k = e} F_{\cDes(w)}
    (x, y)
  \end{equation}
for all positive integers $n, k$ with $k$ odd.
\end{theorem}

\begin{example}
For $n=k=3$, the $k$-roots of the identity element in $B_n$ 
are $(1, 2, 3)$, $(2, 3, 1)$, $(3, 1, 2)$, $(\bar{2}, \bar{3}, 
1)$, $(\bar{2}, 3, \bar{1})$, $(2, \bar{3}, \bar{1})$, 
$(\bar{3}, \bar{1}, 2)$, $(\bar{3}, 1, \bar{2})$ and 
$(3, \bar{1}, \bar{2})$. Computing the signed descent 
compositions of these signed permutations and applying
Theorem~\ref{thm:krootsB}, we find that

\begin{eqnarray*}
  \ch(r^B_{3, 3}) &=& F_{(3)}(x,y) \, + \, F_{(2,1)}(x,y) \, + \,
  F_{(1,2)}(x,y) \, + \, F_{(\bar{2},1)}(x,y) \, +   \\
  & &  F_{(1,\bar{2})}(x,y) \, +\, F_{(1,\bar{1},1)}(x,y) \, + \, 2 F_{(\bar{1},1,\bar{1})}
  (x,y) \, + \, F_{(1,\bar{1},\bar{1})}.
\end{eqnarray*}
The multiset of these signed compositions coincides with the
multiset of signed descent compositions of the standard Young 
bitableaux of shapes $((3),\varnothing)$, $((2,1),\varnothing)$, 
$((1),(2))$ and $((1),(1,1))$. Thus, we may deduce from Proposition~\ref{prop:schurprod} that
\[ \ch(r^B_{3, 3}) \ = \ s_{(3)}(x) \, + \, s_{(2,1)}(x) \, + \,
s_{(1)}(x) s_{(2)}(y) \, + \, s_{(1)}(x) s_{(1,1)}(y).
\]
This yields the decomposition $r^B_{3, 3} = 
\chi^{((3),\varnothing)} + \chi^{((2,1),\varnothing)} +
\chi^{((1),(2)} + \chi^{((1),(1,1))}$.
\end{example}

The proof of Theorem~\ref{thm:krootsB} follows the computation 
of $\ch(r_{n, k})$ in \cite{Thi92}, as presented in the solution to
\cite[Exercise~7.69~(c)]{StaEC2}. We will write $L^B_n (x, y)$ for
the formal sum of the evaluations of all primitive necklaces of
length $n$ over the alphabet $\aA$ having an even number of $y$
variables (i.e. for the function $L^B_{\alpha, \beta} (x, y)$ when
$\alpha = (n)$ and $\beta = \varnothing$). We will also write $L_n
(x, y)$ for the corresponding formal sum with no restriction on
the number of $y$ variables. We will first establish the following
lemma.

\begin{lemma} \label{lem:gen}
For every odd positive integer $k$
  $$ \sum_{d \, | k} \sum_{n \ge 1} \, \frac{1}{n} L^B_d (x^n, y^n)
     \, t^{nd}  = 
     \sum_{n \ge 1} \, \frac{1}{2n} \left( (p^+_{n/(n, k)}
     (x, y))^{(n,k)}  +  (p^-_{n/(n, k)} (x, y))^{(n,k)} \right)
     t^n, $$
where $x^n = (x^n_1, x^n_2,\dots)$, $y^n = (y^n_1, y^n_2,\dots)$
and $(n, k)$ denotes the greatest common divisor of $n$ and $k$.
\end{lemma}
\begin{proof}
The definition of the functions $L_d (x, y)$ and $L^B_d (x, y)$ as
generating functions of necklaces implies that $L^B_d (x, y) =
(L_d (x, y) + L_d (x, -y))/2$. Since
  \begin{equation} \label{eq:Lgen1}
    \sum_{d \, | k} \, \sum_{n \ge 1} \frac{1}{n} L_d (x^n, y^n)
    \, t^{nd} \ = \ \sum_{n \ge 1} \, \frac{1}{n} (p^+_{n/(n, k)}
    (x, y))^{(n,k)} \, t^n
  \end{equation}
follows from \cite[Equation~(7.216)]{StaEC2}, it suffices to
verify that
  \begin{equation} \label{eq:Lgen2}
    \sum_{d \, | k} \sum_{n \ge 1} \, \frac{1}{n} L_d (x^n, -y^n)
    \, t^{nd} \ = \ \sum_{n \ge 1} \, \frac{1}{n} (p^-_{n/(n, k)}
    (x, y))^{(n,k)} \, t^n.
  \end{equation}
We sketch the proof of this equation, which is similar to that of
(\ref{eq:Lgen1}). Let us denote by $\mu: \{1, 2,\dots\} \to \ZZ$
the number theoretic M\"obius function. The following well known
(see, for instance, \cite[Equation~2.2]{GR93} or
\cite[Exercise~7.89~(a)]{StaEC2}) formula
  $$ L_m (x) \ = \ \frac{1}{m} \sum_{d \, | \, m} \mu(d)
                   (p_d (x))^{m/d} $$
can be proved by a simple M\"obius inversion argument. Replacing
$x$ by $(x^{n/m}, -y^{n/m})$ in this equation and assuming that
$m$ is odd gives
  $$ mL_m (x^{n/m}, -y^{n/m}) \ = \ \sum_{d \, | \, m} \mu(d)
     (p^-_{nd/m} (x, y))^{m/d} \ = \ \sum_{d \, | \, m} \mu(m/d)
     (p^-_{n/d} (x, y))^d $$
whenever $n$ is divisible by $m$. Applying M\"obius inversion, we
get
  $$ (p^-_{n/m} (x, y))^m \ = \ \sum_{d \, | \, m} dL_d
     (x^{n/d}, -y^{n/d}) $$
whenever $m$ is odd and $n$ is divisible by $m$. Since $k$ is
assumed to be odd, we may replace $m$ by $(n, k)$ in the last
equation to get
  $$ (p^-_{n/(n, k)} (x, y))^{(n, k)} \ = \ \sum_{d \, | \, k}
     \sum_{d \, | \, n} dL_d (x^{n/d}, -y^{n/d}) $$
for every positive integer $n$. Multiplying by $t^n/n$ and summing
over for all $n \ge 1$ we get (\ref{eq:Lgen2}) and the proof
follows.
\end{proof}

\begin{proof}[Proof of Theorem~\ref{thm:krootsB}] From the
defining equation (\ref{eq:chB}) of the Frobenius characteristic
map we have
  $$ \ch(r^B_{n, k}) \ = \ \frac{1}{2^n n!} \, \sum_{w \in B_n}
     r^B_{n, k} (w) p_w(x, y) \ = \ \frac{1}{2^n n!} \, \sum_{u \in
     B_n} p_{u^k}(x, y). $$
As in the solution to \cite[Exercise~7.69~(c)]{StaEC2} we observe
that the $k$th power of a positive $m$-cycle in $B_n$ is a product
of $(m, k)$ disjoint positive cycles of length $m/(m, k)$ and,
since $k$ is odd, the analogous statement holds for negative
cycles. An application of the exponential formula
\cite[Corollary~5.1.9]{StaEC2} then yields
  \begin{equation} \label{eq:rgen}
    \sum_{n \ge 0} \ch(r^B_{n, k}) t^n \ = \ \exp \,
    \sum_{n \ge 1} \left( (p^+_{n/(n, k)} (x, y))^{(n,k)} \, + \,
    (p^-_{n/(n, k)} (x, y))^{(n,k)} \right) \frac{t^n}{2n}.
  \end{equation}

We now denote by $G_{n,k}(x, y)$ the right-hand side of
(\ref{eq:krootsB}) and compute its generating function as follows.
Since $k$ is odd, a signed permutation $w \in B_n$ is a $k$-root
of the identity element if and only if the cycle decomposition of
$w$ involves only positive cycles whose lengths divide $k$. Thus,
it follows from Theorem~\ref{thm:poirier} that
  $$ \sum_{n \ge 0} G_{n,k}(x, y) t^n \ = \ \sum_{\alpha} L^B_{\alpha,
      \varnothing} (x, y) t^{|\alpha|} $$
where the sum ranges over all integer partitions $\alpha$ such
that every part of $\alpha$ divides $k$. The definition of
$L^B_{\alpha, \beta} (x, y)$ in terms of multisets of necklaces
and that of plethysm of symmetric functions in turn imply that
  \begin{equation} \label{eq:Ggen}
    \sum_{n \ge 0} G_{n,k}(x, y) t^n \ = \ \prod_{d \, | \, k}
    h[L^B_d (tx, ty)]
  \end{equation}
where $h(x) = \sum_{n \ge 0} h_n(x)$ is the formal sum of all
complete homogeneous symmetric functions $h_n(x)$. The basic
formula $\log h(x) = \log \prod_{i \ge 1} 1/(1-x_i) = \sum_{n \ge
1} p_n (x)/n$ and (\ref{eq:Ggen}) then imply that
  \begin{equation} \label{eq:Ggen2}
    \sum_{n \ge 0} G_{n,k}(x, y) t^n \ = \ \exp \, \sum_{d \, | \, k}
    \sum_{n \ge 1} \, \frac{1}{n} L^B_d (x^n, y^n) \, t^{nd}.
\end{equation}
Comparing (\ref{eq:rgen}) to (\ref{eq:Ggen2}) and using
Lemma~\ref{lem:gen} gives $\ch(r^B_{n, k}) = G_{n,k}(x, y)$ for
all $n\ge 0$ and the proof follows. 
\end{proof}

\section{Arc permutations}
\label{sec:arc}

Arc permutations, originally introduced in the study of
triangulations~\cite{AR12}, have interesting combinatorial
properties. For instance, they can be characterized by pattern
avoidance, they carry interesting graph and poset structures, as
well as an affine Weyl group action, and afford well factorized
unsigned and signed enumeration formulas~\cite{ER14}. Two type $B$
extensions were introduced in~\cite{ERxx}. This section reviews
the relevant definitions and shows that one of these extensions is
a fine set for an $\mathfrak{S}_n$-character, while the other is a
fine set for a $B_n$-character.

A permutation $w \in \mathfrak{S}_n$ is said to be an \emph{arc
permutation} if, for every $i \in [n]$, the set $\{w(1), w(2),\dots,w(i)\}$ is an interval in
$\ZZ_n$ (where the letter $n$ is identified with zero). For example, $(2, 1, 5, 3, 4)$ is an arc permutation
in $\mathfrak{S}_5$ but $(2, 1, 5, 6, 3, 4)$ is not an arc
permutation in $\mathfrak{S}_6$. The set of arc permutations in
$\mathfrak{S}_n$ is denoted by $\aA_n$.

As will be explained in the sequel, the following theorem can be
deduced from results of Elizalde and Roichman~\cite{ER14}. This
theorem will be extended to type $B$ in this section. Let
$V_{n-1}$ be an $(n-1)$-dimensional vector space over $\CC$, on
which $\mathfrak{S}_{n-1}$ acts by permuting coordinates.
\begin{theorem} \label{thm:arcA}
For $n \ge 2$, the set $\aA_n$ is fine for the character of the
induced representation $\bigwedge V_{n-1}
\uparrow^{\mathfrak{S}_n}$ of the exterior algebra of $V_{n-1}$
from $\mathfrak{S}_{n-1}$ to $\mathfrak{S}_n$.

Equivalently, $\sum_{w \in \aA_n} F_{n, \Des(w)}(x) = \ch(\chi)$
where
\begin{equation*}
  \chi = \sum\limits_{k=1}^{n-1} \chi^{(k, 1^{n-1-k})}
  \uparrow_{\mathfrak{S}_{n-1}}^{\mathfrak{S}_n} 
  = \chi^{(n)} \, + \, \chi^{(1^n)} \, + \,
      2 \sum\limits_{k=2}^{n-1} \chi^{(k, 1^{n-k})} \, + \,
      \sum\limits_{k=2}^{n-2} \chi^{(k, 2, 1^{n-k-2})}.
\end{equation*}
\end{theorem}

Two type $B$ extensions of the concept of an arc permutation,
suggested in~\cite{ERxx}, can be described as follows. Let us
identify $\Omega_n$ with $\ZZ_{2n}$ by associating $\bar{i} \in
\Omega_n$ with $n + i \in\ZZ_{2n}$ (and thus $\bar{n} \in
\Omega_n$ with $0\in\ZZ_{2n}$) for every $i \in
[n]$. A signed permutation $w \in B_n$ is said to be a
\emph{$B$-arc permutation} if $\{w(i), w(i+1),\dots,w(n)\}$ is an
interval in $\ZZ_{2n}$ for every $i \in [n]$. For instance,
$(\bar{2}, 3, \bar{1}, 5, 4)$ is a $B$-arc permutation in $B_5$.
The set of $B$-arc permutations in $B_n$ will be denoted by
$\AB_n$. Note that $|\AB_n|=n2^n$.

We also say that $w \in B_n$ is a \emph{signed arc
permutation} if, for each $i \in \{2,\dots,n-1\}$,
\begin{itemize}
\item the set $\{|w(1)|,\dots,|w(i-1)|\}$ is an interval in
$\ZZ_n$; and \item $w(i)$ is unbarred if $|w(i)|-1 \in
\{|w(1)|,\dots,|w(i-1)|\}$ and barred if $|w(i)| + 1 \in
\{|w(1)|,\dots,|w(i-1)|\}$ (with addition in $\ZZ_n$).
\end{itemize}
For instance, $(\bar{3}, \bar{2}, 4, 1)$ is a signed arc
permutation in $B_4$. The set of signed arc permutations in $B_n$
will be denoted by $\aA^s_n$. Since there is no restriction on the
signs of $w(1)$ and $w(n)$, the number of signed arc permutations
in $B_n$ satisfies $|\aA^s_n| = 4|\aA_n| = n2^n$ for $n \ge 2$.

\medskip
The main results of this section state that $\AB_n$ and $\aA^s_n$
are fine sets for characters of $\mathfrak{S}_n$ and $B_n$,
respectively, which are described explicitly. We recall 
from Section~\ref{subsec:perm} that the descent set of $w \in B_n$ is defined as
$\Des(w) = \{ i \in [n-1]: w(i) >_r w(i+1)\}$, and that its signed descent set
$\cDes(w)$ is given by Definition~\ref{cDes}. The group $B_{n-1}$
acts on the vector space $V_{n-1}$ by permuting coordinates and switching their signs.

\begin{theorem} \label{thm:B-arc}
For every $n \ge 2$,
\[ \sum\limits_{w \in \AB_n} F_{n, \Des(w)} (x) \ = \
\ch \left( \, 2 \sum\limits_{k=1}^{n-1} \chi^{(k,1^{n-k-1})}
\uparrow^{\mathfrak{S}_n}_{\mathfrak{S}_{n-1}} \, + \ n
\sum\limits_{k=1}^{n} \chi^{(k,1^{n-k})} \right). \]
\end{theorem}

\begin{theorem} \label{thm:sign-arc}
For every $n \ge 2$, the set $\aA^s_n$ is fine for the
$B_n$-character induced from the exterior algebra $\bigwedge
V_{n-1}$. Equivalently,
\[ \sum\limits_{w \in \aA^s_n} F_{\cDes(w)} (x, y) \ = \ \ch \left( \,
   \sum\limits_{k=0}^{n-1} \chi^{(k),(1^{n-k-1})}
   \uparrow_{B_{n-1}}^{B_n} \right). \]
\end{theorem}

\begin{example}
We have $$\AB_2 = \aA^s_2 = \{(1,2), (\bar1,2), (1,\bar2), 
(\bar1,\bar2), (2,1), (\bar2,1), (2,\bar1), (\bar2,\bar1)\}.$$ Thus, for
$n=2$, Theorem~\ref{thm:B-arc} states that 
\begin{eqnarray*}4F_{2,\varnothing}(x) \, + \, 4F_{2,\{1\}} (x)  &=& 4s_{(2)}(x)
\, + \, 4s_{(1,1)}(x) \\ &=&  \ch\left(2\chi^{(1)}\uparrow^{\mathfrak{S}_2}_{\mathfrak{S}_{1}} \, + 
\, 2(\chi^{(1,1)}+\chi^{(2)})\right),\end{eqnarray*}
and Theorem~\ref{thm:sign-arc} states that
\begin{multline*}F_{(2)}(x,y) \, + \, 2F_{(1,\bar1)}(x,y) \, + \, 2F_{(\bar1,1)}(x,y) \, + \, F_{(\bar2)}(x,y) \, + \, F_{(1,1)}(x,y) \, + \, F_{(\bar1,\bar1)}(x,y)\\
= \ s_{(2)}(x) \, + \, s_{(2)}(y) \, + \, s_{(1,1)}(x) \, + \, s_{(1,1)}(y) \, + \, 2s_{(1)}(x)s_{(1)}(y) \\ = 
\ \ch\left( \left( \chi^{\varnothing,(1)} \, + \, \chi^{(1),\varnothing}\right) \uparrow_{B_{1}}^{B_2}\right),\end{multline*}
where we have indexed the functions $F_\sigma(x,y)$ by 
the corresponding signed descent compositions.
\end{example}

For the proof of Theorems~\ref{thm:arcA} and~\ref{thm:B-arc} we
need one more definition. A permutation $w \in \mathfrak{S}_n$ is
called \emph{left-unimodal} (respectively, \emph{right-unimodal})
if the set $\{w(1), w(2),\dots,w(n)\}$ (respectively, $\{w(i),
w(i+1),\dots,w(n)\}$) is an interval in $\ZZ$ for every $i \in
[n]$. Clearly, all left-unimodal and all right-unimodal
permutations are arc permutations. The set of left-unimodal
permutations in $\mathfrak{S}_n$ will be denoted by $\lL_n$, and
that of right-unimodal permutations by $\rR_n$.

\begin{proposition} \label{prop:unimodal}
For $n \ge 1$,
\[ \sum\limits_{w \in \lL_n} F_{n, \Des(w)} (x) \ = \
   \sum\limits_{w \in \rR_n} F_{n, \Des(w)} (x) \ = \
   \ch \left( \, \sum\limits_{k=1}^{n} \chi^{(k,1^{n-k})} \right). \]
\end{proposition}
\begin{proof}
As explained in the proofs of Propositions 4 and 6 in \cite{ER14},
from the definition of left and right unimodality we get
\[ \sum\limits_{w \in \lL_n} F_{n, \Des(w)} (x) \ = \
   \sum\limits_{w \in \rR_n} F_{n, \Des(w)} (x) \ = \
   \sum_{S \subseteq [n-1]} F_{n, S} (x). \]
A simple application of Proposition~\ref{prop:sFexpansion} shows
that
\[ \sum_{S \subseteq [n-1]} F_{n, S} (x) \ = \ \sum\limits_{k=1}^{n}
   s_{(k,1^{n-k})} (x) \]
and the proof follows.
\end{proof}

\begin{proof}[Proof of Theorem~\ref{thm:arcA}] We may assume
that $n \ge 4$. Then \cite[Theorem~5]{ER14} may be rephrased as
stating that
\[ \sum\limits_{w \in \aA_n \setminus (\lL_n \cup \rR_n)} F_{n, \Des(w)} (x) \ = \ \ch \left( \, \sum\limits_{k=2}^{n-2} \chi^{(k,2,1^{n-k-2})}
\right). \]
The second statement in the theorem follows from this
equation, Proposition~\ref{prop:unimodal} and the fact that $\lL_n
\cap \rR_n$ consists of the identity permutation and its reverse.
The equivalence to the first statement is explained in the proof
of \cite[Theorem~6]{ER14}. 
\end{proof}

\begin{proof}[Proof of Theorem~\ref{thm:B-arc}] The proposed
equation may be rewritten as
\begin{multline*}
\sum\limits_{w \in \AB_n} F_{n, \Des(w)} (x) \\ = \ \ch \left(
(n+2) (\chi^{(n)} \, + \, \chi^{(1^n)}) \, + \, (n+4)
\sum\limits_{k=2}^{n-1} \chi^{(k, 1^{n-k})} \, + \, 2
\sum\limits_{k=2}^{n-2} \chi^{(k, 2, 1^{n-k-2})} \right).
\end{multline*}

To prove this equation, and in view of Theorem~\ref{thm:arcA} and
Proposition~\ref{prop:unimodal}, it suffices to find a subset
$\wAB_n$ of $\AB_n$ for which there exist an $n$-to-one descent
preserving map from $\wAB_n$ to the set $\rR_n$ of right-unimodal
permutations and a two-to-one descent preserving map from $\AB_n
\sm \wAB_n$ to the set $\aA_n$ of arc permutations. We will verify
these properties when $\wAB_n$ is the set of $w \in \AB_n$ such
that $w(i) = 1$ for some $i \in [n]$. For the first property, we
observe that $\wAB_n$ is the disjoint union of the sets $\wAB_{n,
k}$, consisting of the elements $w \in \AB_n$ for which $(w(1),
w(2),\dots,w(n))$ is a permutation of
$\{\overline{k+1},\dots,\bar{n}, 1,\dots,k\}$, for $1 \le k \le n$
and that there is a natural descent preserving bijection from each
$\wAB_{n, k}$ to $\rR_n$.

For the second property, we note that $\bar{\aA}_n := \AB_n \sm
\wAB_n$ is the set of $w \in \AB_n$ for which $(w(1),
w(2),\dots,w(n))$ is a permutation of $\{k+1,\dots,n,
\bar{1},\dots,\bar{k}\}$ for some $k \in [n]$ and consider the map
$f: \bar{\aA}_n \to \aA_n$ which simply forgets the bars, meaning
that $f(w) = (|w(1)|, |w(2)|,\dots,|w(n)|)$ for $w \in
\bar{\aA}_n$. This map is clearly descent preserving. We leave
to the reader to verify that the preimages of a given arc
permutation $(a_1, a_2,\dots,a_n) \in \aA_n$ under $f$ are
\begin{itemize}
\item $(\bar{1}, a_2,\dots,a_n)$ and $(\bar{1},
\bar{a}_2,\dots,\bar{a}_n)$, if $a_1 = 1$, and \item $(a_1,
b_2,\dots,b_n)$ and $(\bar{a}_1, b_2,\dots,b_n)$ otherwise, where
$b_i = a_i$ if $a_i > a_1$ and $b_i = \bar{a}_i$ if $a_i < a_1$,
for $i \ge 2$.
\end{itemize}
This shows that the map $f$ is also two-to-one.
\end{proof}

\begin{remark} \label{rem:B-arc} \rm
Using the reasoning in Remark~\ref{rem:coinv}, it can be verified
that $\AB_n$ is not a fine set for any $B_n$-character for $n \ge
3$.
\end{remark}

Even though $\aA^s_n$ is not a union of Knuth classes of type $B$,
the Robinson--Schensted correspondence of type $B$ will be used in
the proof of Theorem~\ref{thm:sign-arc}. We recall from
Section~\ref{subsec:perm} that we think of signed permutations $w
\in B_n$ as written in the form $(w(1), w(2),\dots,w(n))$.

\begin{proof}[Proof of Theorem~\ref{thm:sign-arc}] The proposed
equation may be rewritten as
\begin{eqnarray*}
  \sum\limits_{w \in \aA^s_n} F_{\cDes(w)} (x, y) &=& s_{(n)}(x) \
  + \ s_{(1^n)}(y) \ + \ \sum\limits_{k=1}^{n-1} s_{(k, 1)} (x)
  s_{(1^{n-k-1})} (y) \ + \\
  & & \sum\limits_{k=1}^{n-1} s_{(k-1)} (x) s_{(2, 1^{n-k-1})} (y)
  \ + \ 2 \sum\limits_{k=1}^{n-1} s_{(k)} (x) s_{(1^{n-k})} (y).
\end{eqnarray*}

By Corollary~\ref{cor:main} (or Proposition~\ref{prop:schurprod}),
it suffices to show that the distribution of $\cDes$ over
$\aA^s_n$ is equal to its distribution over all standard Young
bitableaux of the following multiset of shapes:
\begin{enumerate}[(i)]
\item shapes $((n),\varnothing)$ and $(\varnothing, (1^n))$
with multiplicity one; \item all shapes $((k,1),
(1^{n-k-1}))$ and $((n-k-1),(2,1^{k-1}))$, for $1\le k\le n-1$,
with multiplicity one; and \item all shapes
$((k),(1^{n-k}))$, for $1\le k \le n-1$, with multiplicity two.
\end{enumerate}

Given our discussion of the Robinson--Schensted correspondence for
$B_n$ in Section~\ref{sec:Knuth} (and, in particular, since the
map $Q^B(\cdot)$ preserves the signed descent set), it suffices to
show that the restriction of $Q^B(\cdot)$ to $\aA^s_n$ induces a
two-to-one map from a subset $\widehat\aA^s_n$ of $\aA^s_n$ to the
set of standard Young bitableaux of shapes of type (iii) and a
bijection from $\aA^s_n \sm \widehat\aA^s_n$ to the set of
standard Young bitableaux of shapes of types (i) and (ii).

Let $e \in B_n$ be the identity permutation and $w_\circ \in B_n$
be defined by $w_\circ(i) = \overline{n-i+1}$ for $i \in [n]$. We
note that $\aA^s_n$ is invariant under right multiplication by
$w_\circ$. For $k \in [n-1]$, we denote by $\aA_{n, k}$ the set
consisting of all signed permutations which are shuffles of the
sequences $(\bar{k}, \overline{k-1},\dots,\bar{1})$ and
$(k+1,\dots,n)$. Then $\aA_{n, k}\subseteq \aA^s_n$ and $\aA_{n,
k} w_\circ$ consists of all shuffles of
$(\bar{n},\overline{n-1},\dots,\overline{k+1})$ with $(1,
2,\dots,k)$. Clearly, $\aA_{n, k}$ and $\aA_{n, k}w_\circ$ are
disjoint for $k \in [n-1]$ and $\aA_{n, i} \sqcup \aA_{n,
i}w_\circ$ and $\aA_{n, j} \sqcup \aA_{n, j}w_\circ$ are disjoint
for $i \ne j$. We choose $\widehat\aA^s_n$ to be the disjoint
union
\[ \widehat\aA^s_n \ := \ \bigsqcup_{k=1}^{n-1} \, (\aA_{n, k} \sqcup \aA_{n, k} w_\circ). \]
For $w \in \aA_{n, k}$, we note that $Q^B(w)$ is the standard
Young bitableau of shape $((n-k),(1^k))$ whose entries in $(1^k)$
are the positions of the barred letters in $w$. As a result,
$Q^B(\cdot)$ induces a bijection from $\aA_{n, k}$ to
$\SYT((n-k),(1^k))$. Similarly, $Q^B(\cdot)$ induces a bijection
from $\aA_{n, k} w_\circ$ to $\SYT((k),(1^{n-k}))$. Thus,
$Q^B(\cdot)$ induces a two-to-one map from $\widehat\aA^s_n$ to
the set of standard Young bitableaux of shapes of type (iii).

Finally, we need to show that $Q^B(\cdot)$ induces a bijection
from $\aA^s_n \sm \widehat\aA^s_n$ to the set of all standard
Young bitableaux of shapes of types (i) and (ii). Let us denote by
$\cC_{n, k}$ the set of all shuffles of $(n-k+1,\dots,n,1)$ with
$(\overline{n-k},\overline{n-k-1},\dots,\bar{2})$ and note that
$\cC_{n, k}$ is a subset of $\aA^s_n$. Denoting by $c$ the
positive cycle $ (1 \ 2 \ \cdots \ n)(\bar{1} \ \bar{2} \ \cdots \
\bar{n})$ and letting
\[ \bB_{n, k} \ := \ \bigsqcup_{r=0}^{k-1} c^r \cC_{n, k}, \]
we may express $\aA^s_n \sm \widehat\aA^s_n$ as a disjoint union
\[ \aA^s_n \sm \widehat \aA^s_n \ = \ \{e, w\} \, \sqcup \,
\bigsqcup_{k=1}^{n-1} \bB_{n, k} \, \sqcup \,
\bigsqcup_{k=1}^{n-1} \bB_{n, k} w_\circ. \]
Clearly, $Q^B(e)$ is the only standard Young bitableau of shape
$((n), \varnothing)$ and $Q^B(w_\circ)$ is the only standard Young
bitableau of shape $(\varnothing, (1^n))$. We leave to the reader
to verify that $Q^B(\cdot)$ induces bijective maps from $\bB_{n,
k}$ and $\bB_{n, k}w_\circ$ to the set of standard Young
bitableaux of shape $((k,1), (1^{n-k-1}))$ and $((n-k,1), (2,
1^{k-1}))$, respectively. This completes the proof.
\end{proof}

\section{Remarks}
\label{sec:rem}

Character formulas for the hyperoctahedral group were studied in
this paper. A generalization to the corresponding Iwahori-Hecke
algebra would be most desirable.

\begin{problem}
Find an appropriate $q$-analogue of Definition~\ref{def:fineB}
which will provide character formulas for the Iwahori--Hecke
algebra of type $B$.
\end{problem}

Most of the results in this paper may be naturally generalized to
the wreath product group $G(r,n) := \ZZ_r\wr S_n$ of $r$-colored
permutations, where $r$ is an arbitrary positive integer (this
paper has focused on the special case of $G(2,n)=B_n$). For
example, the $B_n$-analogue of the fundamental quasisymmetric
functions, considered here, was introduced and studied in the
general setting of $r$-colored permutations by
Poirier~\cite{Poi98}. However, it should be noted that generalizations to $G(r,n)$ are not always straightforward. The
following problem is challenging.

\begin{problem}
Generalize the concept of signed arc permutations and
Theorem~\ref{thm:sign-arc} to the groups $G(r,n)$.
\end{problem}

The situation is more complicated for general complex reflection
groups.

\begin{problem}
Generalize the setting and results of this paper to the classical
complex reflection groups $G(r,p,n)$.
\end{problem}

For example, the generalization of Theorems~\ref{thm:krootsA}
and~\ref{thm:krootsB} is related to the problem of characterizing
the finite groups $G$ for which the function $r_k: G \rightarrow
\NN$, assigning to $w \in G$ the number of $k$-roots of $w$, is a
non-virtual character; see, for instance, \cite[Chapter~4]{Isa76} \cite{Sch91b}.

The group of even signed permutations $D_n = G(2,2,n)$ is of
special interest.

\medskip

Another possible direction is to extend the setting of this paper
to other (not necessarily classical or finite) Coxeter groups.

\begin{problem}
Generalize the concept of fine set to arbitrary Coxeter groups.
\end{problem}

This objective may require a different approach. For example, we
believe that the results in Section~\ref{sec:flag-inv} may be
generalized to classical complex reflection groups. This hope is
based on the current good understanding of the flag-major index
and its role in the combinatorics of the coinvariant and diagonal
invariant algebras of classical complex reflection groups; see,
for instance,~\cite{Ca11}. Unfortunately, exceptional Weyl groups
are still mysterious and the problem of finding a ``correct"
flag-major index on these groups is widely open.

\section*{Acknowledgements} S.E. was partially supported by Simons 
Foundation grant \#280575 and NSA grant H98230-14-1-0125.

\end{document}